\documentclass{article}

\usepackage{extarrows}
\usepackage{amsthm}
\usepackage{latexsym}
\usepackage[american]{babel}
\usepackage{times}
\usepackage{tikz}
\usepackage{tikz-cd}
\usepackage{comment}
\usepackage{url}
\usepackage{mathrsfs}
\usepackage{amsmath,amstext,amsthm,amssymb,amsfonts,amscd}
\usepackage{mathrsfs}
\usepackage[colorlinks=true]{hyperref}
\usepackage{graphicx}
\usepackage[T1]{fontenc}
\usepackage[latin1]{inputenc}
\usepackage{color,tocvsec2}
\usepackage{typearea}
\usepackage{charter}
\usepackage{enumerate}
\usepackage{lscape}
\usepackage[all]{xy}
\usepackage[normalem]{ulem}

\RequirePackage[T1]{fontenc}
\title{Deformations of cohesive modules on compact complex manifolds}
\author{ Zhaoting Wei\\ Texas A\& M University-Commerce\\ zhaoting.wei@tamuc.edu}
\date{}

\begin{document}
\newcommand{\End}{\text{End}}
\newcommand{\HBC}{H_\text{BC}}
\newcommand{\chBC}{\text{ch}_\text{BC}}
\newcommand{\Hom}{\text{Hom}}
\newcommand{\ad}{\text{ad}}
\newcommand{\id}{\text{id}}
\newcommand{\str}{\text{Tr}_{\text{s}}}
\newcommand{\pr}{\text{pr}}
\newcommand{\Ad}{\text{Ad}}
\newcommand{\Vect}{\text{Vect}}
\newcommand{\CoAd}{\text{CoAd}}
\newcommand{\coad}{\text{coad}}
\newcommand{\Pol}{\text{Pol}}
\newcommand{\Cl}{\text{Cl}}
\newcommand{\As}{\text{As}}
\newcommand{\ch}{\text{ch}}
\newcommand{\Lied}{\text{L}}
\newcommand{\Ext}{\text{Ext}}
\newcommand{\ZZ}{\mathbb{Z}/2\mathbb{Z}}
\newcommand{\bHom}{\mathbb{H}\text{om}}
\newcommand{\bCone}{\underline{\mathbb{C}\text{one}}}
\newcommand{\Image}{\text{Image }}
\newcommand{\Acyc}{\text{Acyc}}
\newcommand{\Cinfty}{C^{\infty}}
\newcommand{\gb}{\text{gb}}
\newcommand{\diff}{\text{d}}
\newcommand{\Diff}{\text{Diff}}
\newcommand{\dpar}{\partial}
\newcommand{\dbar}{\overline{\partial}}
\newcommand{\pprime}{\prime\prime}
\newcommand{\coh}{\text{coh}}
\newcommand{\KS}{\text{KS}}
\newcommand{\ku}{\text{ku}}

\newtheorem{thm}{Theorem}[section]
\newtheorem{lemma}[thm]{Lemma}
\newtheorem{prop}[thm]{Proposition}
\newtheorem{coro}[thm]{Corollary}
\newtheorem{defi}{Definition}[section]
\newtheorem{eg}{Example}[section]
\newtheorem{rmk}{Remark}[section]

\numberwithin{equation}{section}

\maketitle

\begin{abstract}
Cohesive modules give a dg-enhancement of the bounded derived category of coherent sheaves on a complex manifold via superconnections. In this paper we discuss the deformation theory of cohesive modules on compact complex manifolds. This generalizes the deformation theory of holomorphic vector bundles and coherent sheaves. We also develop the theory of Kuranishi maps and obstructions of deformations of cohesive modules and give some examples of unobstructed deformations.


\end{abstract}

\section{Introduction}
In \cite{block2010duality} Block introduced the concept of \emph{cohesive modules}.  For a compact complex manifold $X$, a cohesive module $\mathcal{E}$ on $X$ consists of a cochain complex of $C^{\infty}$ vector bundles $E^{\bullet}$ together with a flat $\dbar$-superconnection $A^{E^{\bullet}\prime\prime}$. Block proved in \cite{block2010duality} that cohesive modules on $X$ form a dg-category $B(X)$, which gives a dg-enhancement of $D^b_{\text{coh}}(X)$, the bounded derived category of coherent sheaves on $X$.  Moreover, \cite{chuang2021maurer} generalizes the result in \cite{block2010duality} to the case that $X$ is non-compact with a slightly more restricted definition of coherent sheaves.  See Section \ref{section: cohesive modules} for a quick review and \cite{bismut2021coherent} for some applications.

One of the advantages of the dg-category of cohesive modules $B(X)$ over the derived category $D^b_{\text{coh}}(X)$ is that any quasi-isomorphism in $B(X)$ has a homotopy inverse.
In this paper we develop  the deformation theory of cohesive modules over a compact complex manifold $X$. 

We first review the classic deformation theory of holomorphic vector bundles as in \cite[Section 3]{narasimhan1982deformations} and \cite[Section 6.4.1]{donaldson1990geometry}. 
Let $E$ be a finite dimensional holomorphic vector bundle on $X$. For a small disk $\Delta\subset \mathbb{C}^m$ which contains $0$, a family of deformations of $E$ over $\Delta$ is a holomorphic vector bundle $F$ over $X\times \Delta$ together with a biholomorphic isomorphism 
\begin{equation}
\phi: F|_{X\times \{0\}}\overset{\simeq}{\to} E.
\end{equation}
Two deformations $(F,\phi)$ and $(G,\psi)$ are called equivalent if there exists a biholomorphic isomorphism $\theta: F\overset{\simeq}{\to} G$ such that $\psi\circ\theta|_{X\times \{0\}}=\phi$.

By Koszul-Malgrange theorem \cite{koszul1958certaines}, a holomorphic vector bundle is nothing but a $C^{\infty}$-vector bundle together with a flat $\dbar$-connection. Since $X$ is compact, by Ehresmann's fibration theorem, for $\Delta$ sufficiently small, $F$ is isomorphic to $E$ as a $C^{\infty}$-vector bundle. Hence the deformation only changes the $\dbar$-connection.   Let $\dbar_E$ be the $\dbar$-connetion on $E$. Then a family of deformations of $E$ over $\Delta$ is a family of $\dbar$-connections $\dbar_E+\epsilon(t)$ with
\begin{equation}
\epsilon(t)\in \Omega^{0,1}(X, \End(E)) \text{ for each }t\in \Delta
\end{equation}
such that $\epsilon(t)$ depends holomorphically on $t$ and $\epsilon(0)=0$. Moreover, $(\dbar_E+\epsilon(t))^2=0$ gives
\begin{equation}\label{eq: Maurer-Cartan equation for holomorphic vector bundle}
\dbar_E(\epsilon(t))+\frac{1}{2}[\epsilon(t),\epsilon(t)]=0.
\end{equation}
If we consider the \emph{differential graded Lie algebra (DGLA)}
\begin{equation}\label{eq: dgla for holomorphic vector bundle}
L_E:=(\Omega^{0,\bullet}(X, \End(E)), [-,-], \dbar_E)
\end{equation}
where $[-,-]$ is the supercommutator,
then \eqref{eq: Maurer-Cartan equation for holomorphic vector bundle} means that $\epsilon(t)$ is a Maurer-Cartan element in $L_E$. We can further show that equivalent deformations correspond to gauge equivalent Maurer-Cartan elements in $L_E$.

 Inspired by the  deformation theory of holomorphic vector bundles, we have the following definition.
\begin{defi}\label{defi: deformation of cohesive modules in introduction}[See Definition \ref{defi: deformation of cohesive modules} below]
For a cohesive module $\mathcal{E}\in B(X)$, a family of deformation of $\mathcal{E}$ over $\Delta$ is a cohesive module $\mathfrak{F}$
on $X\times \Delta$ together with a homotopy equivalence 
\begin{equation}\label{eq: deformation restrict to 0}
\phi: \mathfrak{F}|_{X\times \{0\}}\overset{\sim}{\to} \mathcal{E}.
\end{equation}

Two families of deformations $(\mathfrak{F},\phi)$ and $(\mathfrak{G},\psi)$ of  $\mathcal{E}$ over $\Delta$ are  called equivalent if there exists a homotopy equivalence $\theta: \mathfrak{F}\to \mathfrak{G}$ such that
\begin{equation}
\psi\circ \theta|_{X\times \{0\}}=\phi \text{ up to chain homotopy.}
\end{equation}
\end{defi}

We denote  the set of equivalent classes of  deformations of $\mathcal{E}$  over $\Delta$ by $\text{Def}_{\Delta}(\mathcal{E})$. It is easy to see that $\text{Def}_{\Delta}(\mathcal{E}_1)$ and $\text{Def}_{\Delta}(\mathcal{E}_2)$ are bijective if $\mathcal{E}_1$ and $\mathcal{E}_2$ are homotopy equivalent. Hence Definition \ref{defi: deformation of cohesive modules in introduction} actually gives a deformation theory of objects in  $D^b_{\text{coh}}(X)$.

To further study  $\text{Def}_{\Delta}(\mathcal{E})$, we introduce the following DGLA which generalizes the $L_E$ in \eqref{eq: dgla for holomorphic vector bundle}.
\begin{equation}\label{eq: dgla LE in introduction}
L_{\mathcal{E}}:=(\Omega^{0,\bullet}(X,\End(E^{\bullet})),[-,-], \diff_{\mathcal{E}}).
\end{equation}  
See Section \ref{subsection: DGLA LE} below for details.

 However, when trying to relate deformations of $\mathcal{E}$ to Maurer-Cartan elements in $L_{\mathcal{E}}$, we find two problems which did not occur for holomorphic vector bundles:
\begin{enumerate}
\item The superconnection $\mathfrak{F}$ has higher terms in both the $X$ and $\Delta$ direction. Therefore $\mathfrak{F}$ is in general not a holomorphic family of cohesive modules over $\Delta$.
\item The morphism $\phi$ is invertible \emph{up to homotopy}, which means that $\mathfrak{F}|_{X\times \{0\}}$ and $\mathcal{E}$ are not the same cohesive module.
\end{enumerate}

We solve  Problem 1 in Section \ref{subsection: deformations and regular deformations} by introducing the concept of regular deformations and proving that every deformation is equivalent to a regular one. We solve Problem 2 in Section \ref{subsection: strong deformations} by introducing the concept of strong deformations and proving that every deformation is equivalent to a strong deformation. Then we prove the following theorem:

\begin{thm}\label{thm: deformation and Maurer-Cartan in introduction}[See Theorem \ref{thm: deformation and Maurer-Cartan} below]
For $\Delta$ sufficiently small, there is a bijection between the set of equivalent classes of deformations of $\mathcal{E}$ over $\Delta$ and gauge equivalent classes of Maurer-Cartan elements in $L_{\mathcal{E}}\otimes \mathcal{M}(\Delta)$, where $L_{\mathcal{E}}\otimes \mathcal{M}(\Delta)$ denotes holomorphic families of elements in $L_{\mathcal{E}}$ which vanish at $0\in \Delta$.
\end{thm}

With Theorem \ref{thm: deformation and Maurer-Cartan in introduction}, we can apply analytic method as in \cite[Section 7]{kobayashi2014differential} to further study  deformations of cohesive modules. In particular, we can define the slice space in $L_{\mathcal{E}}\otimes \mathcal{M}(\Delta)$, construct the Kuranishi map, and develop the obstruction theory. We also give some examples of cohesive modules whose deformations are unobstructed. See Section \ref{section: analytic theory} below.

Notice that the method in our paper is differential geometric, which is parallel to the \v{C}ech-theoretic method in \cite{fiorenza2012differential} and the algebraic method in \cite{lieblich2006moduli}.

This paper is organized as follows: In Section \ref{subsection: cohesive modules} we review cohesive modules on compact manifolds. In Section \ref{subsection: DGLA LE} we introduce the DGLA $L_{\mathcal{E}}$ for a cohesive module $\mathcal{E}$. In Section \ref{subsection: cohesive and coherent} we recall the results on the relation between cohesive modules and coherent sheaves. In Section \ref{subsection: pull backs} we discuss the pull-back of cohesive modules by morphism between noncompact complex manifolds. 

In Section \ref{subsection: deformations and regular deformations} we introduce the definition of deformations and regular deformations of cohesive modules and study their  properties. In Section \ref{subsection: strong deformations} we introduce strong deformations, which are holomorphic families of cohesive modules modeled by $\mathcal{E}$. We then prove any deformation is equivalent to a strong deformation and establish the relation between deformations and Maurer-Cartan elements in the  DGLA $L_{\mathcal{E}}\otimes \mathcal{M}(\Delta)$.

In Section \ref{section: infinitesimal deformations} we study infinitesimal deformations. We show that they are $1$-$1$ correspondence to $\Hom_{\underline{B}(X)}(\mathcal{E},\mathcal{E}[1])$.

In Section \ref{subsection: metrics and Laplacians} we introduce Hermitian metrics on cohesive modules and construct the corresponding Laplacians, which generalize the $\dbar$-Laplacian on a holomorphic vector bundle.  In Section \ref{subsection: Sobolev spaces} we review Sobolev spaces and construction the Green operator associated with the Laplacian defined in Section \ref{subsection: metrics and Laplacians}. In Section \ref{subsection: the slice space} we study the slice space in $L_{\mathcal{E}}\otimes \mathcal{M}(\Delta)$. In Section \ref{subsection: Kuranishi map} we introduce and study the Kuranishi map. In Section \ref{subsection: obstruction} we study the obstruction of lifting an infinitesimal deformation to a genuine deformation. In Section \ref{subsection: unobstructed} we give some examples of cohesive modules such that the obstructions always vanish.

\section*{Acknowledgment}
The author wants to thank Jonathan Block for very inspiring discussions. The author also wants to thank J.P. Pridham for kindly answering questions on deformation theory.

\section{A review of cohesive modules on complex manifolds}\label{section: cohesive modules}
\subsection{The definition of cohesive modules}\label{subsection: cohesive modules}
We first fix some notations. Let $X$ be a  complex manifold of dimension $n$. Let $TX$ and $\overline{TX}$ be the holomorphic and antiholomorphic tangent bundle. Let  $T_{\mathbb{R}}X$ be the corresponding real tangent bundle and $T_{\mathbb{C}}X=T_{\mathbb{R}}X\otimes_{\mathbb{R}}\mathbb{C}$ be its complexification. We have the decomposition $T_{\mathbb{C}}X=TX\oplus \overline{TX}$.

The concept of cohesive modules is introduced by Block in \cite{block2010duality}. 

\begin{defi}\label{defi: cohesive module}
Let $X$ be a complex manifold. A \emph{cohesive module} on $X$ is a bounded, finite rank, $\mathbb{Z}$-graded, $C^{\infty}$-vector bundle $E^{\bullet}$ on $X$ together with a superconnection with total degree $1$
$$
A^{E^{\bullet}\prime\prime}: \wedge^{\bullet}\overline{T^{*}X}  \times  E^{\bullet}\to \wedge^{\bullet}\overline{T^{*}X}  \times E^{\bullet}
$$
such that $A^{E^{\bullet}\prime\prime}\circ A^{E^{\bullet}\prime\prime}=0$.

In more details, $A^{E^{\bullet}\prime\prime}$ decomposes into
\begin{equation}\label{eq: decomposition of anti super conn}
A^{E^{\bullet}\prime\prime}=v_0+\nabla^{E^{\bullet}\prime\prime}+v_2+\ldots
\end{equation}
where 
$$
\nabla^{E^{\bullet}\prime\prime}: E^{\bullet}\to \overline{T^{*}X}  \times E^{\bullet}
$$
 is a  $\dbar$-connection, and for $i\neq 1$
\begin{equation}\label{eq: cohesive module}
v_i\in C^{\infty}(X,\wedge^{i}\overline{T^{*}X}  \hat{\otimes}  \End^{1-i}(E^{\bullet}))
\end{equation}
is  $C^{\infty}(X)$-linear. Here $\hat{\otimes} $ denotes the graded tensor product.

Cohesive modules on $X$ forms a dg-category  denoted by $B(X)$. In more details, let $\mathcal{E}=(E^{\bullet}, A^{E^{\bullet}\prime\prime})$ and $\mathcal{F}=(F^{\bullet}, A^{F^{\bullet}\prime\prime})$ be two cohesive modules on $X$ where
$$
A^{E^{\bullet}\prime\prime}=v_0+\nabla^{E^{\bullet}\prime\prime}+v_2+\ldots
$$ 
and
$$
A^{F^{\bullet}\prime\prime}=u_0+\nabla^{F^{\bullet}\prime\prime}+u_2+\ldots
$$ 
A morphism $\phi:\mathcal{E}\to \mathcal{F}$ of degree $k$ is given by
\begin{equation}\label{eq: decomposition of morphism}
\phi=\phi_0+\phi_1+\ldots
\end{equation}
where
$$
\phi_i\in C^{\infty}(X,\wedge^{i}\overline{T^{*}X}  \hat{\otimes}  \Hom^{k-i}( E^{\bullet},F^{\bullet}))
$$
is   $C^{\infty}(X)$-linear.

For 
$$
\phi=\alpha \hat{\otimes} u\in  C^{\infty}(X,\wedge^{i}\overline{T^{*}X}  \hat{\otimes}  \Hom^{k-i}( E^{\bullet},F^{\bullet}))$$ 
and 
$$
\psi=\beta\hat{\otimes} v\in C^{\infty}(X,\wedge^{j}\overline{T^{*}X}  \hat{\otimes}  \Hom^{l-j}( F^{\bullet},G^{\bullet})),
$$ 
their composition $\psi\phi$ is defined as
\begin{equation}\label{eq: composition of morphisms in B(X)}
\psi \phi:=(-1)^{(l-j)i}\beta\alpha\hat{\otimes} vu\in C^{\infty}(X,\wedge^{i+j}\overline{T^{*}X}  \hat{\otimes}  \Hom^{k+l-i-j}( E^{\bullet},G^{\bullet}))
\end{equation}

 The differential of $\phi$ is given by 
\begin{equation}\label{eq: differential in B(X)}
d\phi=A^{F^{\bullet}\prime\prime}\phi-(-1)^k\phi A^{E^{\bullet}\prime\prime}.
\end{equation}
More explicitly, the $l$th component  of $d\phi$ is 
$$
(d\phi)_l\in C^{\infty}(X,\wedge^{l}\overline{T^{*}X}  \hat{\otimes}  \Hom^{k-l+1}( E^{\bullet},F^{\bullet}))
$$
which is given by
\begin{equation}\label{eq: differential in B(X) degree l}
(d\phi)_l=\sum_{i\neq 1}\big(u_i\phi_{l-i}-(-1)^k\phi_{l-i}v_i\big)+\nabla^{F^{\bullet}\prime\prime}\phi_{l-1}-(-1)^k\phi_{l-1}\nabla^{E^{\bullet}\prime\prime}.
\end{equation}
\end{defi}

\begin{rmk}
In \cite{bismut2021coherent} cohesive modules are called \emph{antiholomorphic superconnections}.
\end{rmk}

We can define mapping cones and shift in $B(X)$. For a degree zero closed map $\phi: \mathcal{E}\to \mathcal{F}$ where $\mathcal{E}= (E^{\bullet}, A^{E^{\bullet}\prime\prime})$ and $ \mathcal{F}=(F^{\bullet}, A^{F^{\bullet}\prime\prime})$, its mapping cone $(C^{\bullet}, A^{C^{\bullet}\prime\prime}_{\phi})$ is defined by
\begin{equation}
C^{n}=E^{n+1}\oplus F^n
\end{equation}
and
\begin{equation}
A^{C^{\bullet}\prime\prime}=\begin{bmatrix}A^{E^{\bullet}\prime\prime}&0\\ \phi(-1)^{\deg(\cdot)}& A^{F^{\bullet}\prime\prime}\end{bmatrix}.
\end{equation}
The shift of $\mathcal{E}$ is $\mathcal{E}[1]$ where
\begin{equation}
E[1]^{n}=E^{n+1}
\end{equation}
and
$$
A^{E^{\bullet}\prime\prime}[1]=A^{E^{\bullet}\prime\prime}(-1)^{\deg(\cdot)}.
$$
It is clear that they give $B(X)$ a pre-triangulated structure hence its homotopy category $\underline{B}(X)$ is a triangulated category.

For later purpose, we recall the following definition

\begin{defi}\label{defi: homotopy equivalence}
A degree $0$ closed morphism $\phi$ between cohesive modules $\mathcal{E}$  and $\mathfrak{F}$  is called a homotopy equivalence if it induces an isomorphism in the homotopy category $\underline{B}(X)$.
\end{defi}

We will need the  following result.

\begin{prop}\label{prop: homotopy equiv degree 0}
A degree $0$ closed morphism $\phi$ between cohesive modules $\mathcal{E}= (E^{\bullet}, A^{E^{\bullet}\prime\prime})$ and $\mathfrak{F}=(F^{\bullet}, A^{F^{\bullet}\prime\prime})$  is  a homotopy equivalence if and only if its degree $0$ component $\phi^0:  (E^{\bullet},v_0)\to (F^{\bullet},u_0)$ is a quasi-isomorphism of cochain complexes.
\end{prop}
\begin{proof}
See \cite[Proposition 2.9]{block2010duality} or \cite[Proposition 6.4.1]{bismut2021coherent}.
\end{proof}

\subsection{The DGLA $L_{\mathcal{E}}$}\label{subsection: DGLA LE}
Let $\mathcal{E}=(E^{\bullet}, A^{E^{\bullet}\prime\prime})\in B(X)$. For later applications we introduce the following differential graded Lie algebra (DGLA)
\begin{equation}\label{eq: dgla LE}
L_{\mathcal{E}}:=(C^{\infty}(X, \wedge^{\bullet}\overline{T^{*}X}\hat{\otimes} \End^{\bullet}(E^{\bullet})),[-,-], \diff_{\mathcal{E}}).
\end{equation}
The grading is given so that 
\begin{equation}
L^k_{\mathcal{E}}=\bigoplus_{i+j=k}C^{\infty}(X,\wedge^{i}\overline{T^{*}X}\hat{\otimes} \End^j(E^{\bullet})).
\end{equation} 
The bracket is the graded commutator: for $\alpha$, $\beta\in  L_{\mathcal{E}}$, we have
\begin{equation}\label{eq: supercommutator}
[\alpha,\beta]:=\alpha\beta-(-1)^{|\alpha||\beta|}\beta\alpha
\end{equation}
where $\alpha\beta$ and $\beta\alpha$ are the compositions as in \eqref{eq: composition of morphisms in B(X)}.
In more details, if we write
$$
\alpha=\omega\hat{\otimes} M\in C^{\infty}(X,\wedge^{i}\overline{T^{*}X}\hat\otimes \End^{k-i}(E^{\bullet}))
$$
and
$$
\beta=\mu\hat{\otimes} N\in C^{\infty}(X,\wedge^{j}\overline{T^{*}X}\hat\otimes \End^{l-j}(E^{\bullet})),
$$
 then we have
\begin{equation}\label{eq: supercommutator in component}
[\alpha,\beta]=(-1)^{(k-i)j}\omega\mu\hat{\otimes}[M,N].
\end{equation}

 The differential $\diff_{\mathcal{E}}$ is induced by $A^{E^{\bullet}\prime\prime}$ as in \eqref{eq: differential in B(X)}. In more details, let
$$
A^{E^{\bullet}\prime\prime}=v_0+\nabla^{E^{\bullet}\prime\prime}+v_2+\ldots
$$ 
be the decomposition  in \eqref{eq: decomposition of anti super conn}. Let $\phi\in L^k_{\mathcal{E}}$ with decomposition
\begin{equation}\label{eq: decompose of phi}
\phi=\phi_0+\phi_1+\ldots
\end{equation}
where $\phi_i\in C^{\infty}(X,\wedge^{i}\overline{T^{*}X}  \hat{\otimes}  \End^{k-i}( E^{\bullet}))$. By \eqref{eq: differential in B(X) degree l} and \eqref{eq: supercommutator}  we know
\begin{equation}\label{eq: diff E in LE}
(\diff_{\mathcal{E}}\phi)_l=\sum_{i\neq 1}[v_i, \phi_{l-i}]+\nabla^{E^{\bullet}\prime\prime}\phi_{l-1}-(-1)^k\phi_{l-1}\nabla^{E^{\bullet}\prime\prime}.
\end{equation}
As in the standard notation, we denote $\nabla^{E^{\bullet}\prime\prime}\phi_{l-1}-(-1)^k\phi_{l-1}\nabla^{E^{\bullet}\prime\prime}$ by $\nabla^{E^{\bullet}\prime\prime}(\phi_{l-1})$. Then \eqref{eq: diff E in LE} becomes
\begin{equation}\label{eq: diff E in LE shorter}
(\diff_{\mathcal{E}}\phi)_l=\nabla^{E^{\bullet}\prime\prime}(\phi_{l-1})+\sum_{i\neq 1}[v_i, \phi_{l-i}].
\end{equation}

By definition, for each $i$ we have 
\begin{equation}\label{eq: DGLA LE and morphism in B(X)}
\Hom_{\underline{B}(X)}(\mathcal{E},\mathcal{E}[i])=(\ker \diff_{\mathcal{E}}\cap L_{\mathcal{E}}^i)/(\text{Im }\diff_{\mathcal{E}}\cap L_{\mathcal{E}}^i).
\end{equation}

Recall (see \cite[Section 3.6]{keller2005deformation}) that an $L_{\infty}$ morphism $\Phi: \mathfrak{g}\to \mathfrak{h}$ between DGLAs consists of a system of linear maps
\begin{equation}
\Phi_n:\mathfrak{g}^{\otimes n}\to \mathfrak{h} \text{ for }n\geq 1
\end{equation}
such that $\deg \Phi_n=1-n$ and they satisfy the following conditions:
\begin{enumerate}
\item \begin{equation}\label{eq: L infty morphism antisymmetric}
\Phi_n(x_1,\ldots, x_{i+1},x_i,\ldots,x_n)=(-1)^{|x_i||x_{i+1}|}\Phi_n(x_1,\ldots,x_i,x_{i+1},\ldots,x_n);
\end{equation}
\item $\Phi_1\circ d_{\mathfrak{g}}=d_{\mathfrak{h}}\circ \Phi_1;$
\item for each $n\geq 2$
\begin{equation}\label{eq: L infty morphism}
\begin{split}
&\sum_{i<j}\pm \Phi_{n-1}([x_i,x_j], x_1,\ldots \widehat{x_i},\ldots,\widehat{x_j},\ldots, x_n)=\\
&\sum_{1\leq p\leq n-1, ~\sigma\in S_n}\pm\frac{1}{2}[\Phi_p(x_{\sigma(1)},\ldots, x_{\sigma(p)}),\Phi_q(x_{\sigma(p+1)},\ldots, x_{\sigma(n)})]\\
&+ d_{\mathfrak{h}}\Phi_n(x_1,\ldots,x_n)+\sum_{1\leq k\leq n}\pm \Phi_n(x_1,\ldots, d_{\mathfrak{g}}(x_k),\ldots, x_n).
\end{split}
\end{equation}
We refer to  \cite[Section 3.6]{keller2005deformation} for details about the $\pm$ signs in \eqref{eq: L infty morphism}.
\end{enumerate}
Moreover, $\Phi$ is called an $L_{\infty}$ quasi-isomorphism if $\Phi_1: (\mathfrak{g},d_{\mathfrak{g}})\to  (\mathfrak{h},d_{\mathfrak{h}})$ is a quasi-isomomrphism.

\begin{lemma}\label{lemma: homotopy equivalent L infty equivalent of dgla}
If two cohesive modules $\mathcal{E}$ and $\mathcal{F}$ are homotopic equivalent, then the corresponding DGLAs $L_{\mathcal{E}}$ and $L_{\mathcal{F}}$ are $L_{\infty}$ quasi-isomorphic.
\end{lemma}
\begin{proof}
Let $\phi: \mathcal{F}\to \mathcal{E}$ and $\psi: \mathcal{E}\to \mathcal{F}$ be a pair of homotopy equivalence together with a homotopy operator $h\in \Hom^{-1}_{B(X)}(\mathcal{F},\mathcal{F})$ such that
\begin{equation}
\psi\phi-\id_{F^{\bullet}}=A^{F^{\bullet}\prime\prime}h+hA^{F^{\bullet}\prime\prime}.
\end{equation}
Then we define an $L_{\infty}$ morphism $\Phi: L_{\mathcal{F}}\to L_{\mathcal{E}}$ by
\begin{equation}
\Phi_1(x)=\psi x\phi
\end{equation}
and for $n\geq 2$
\begin{equation}\label{eq: L infty morphism for cohesive module equivalence}
\Phi_n(x_1,\ldots, x_n)=\sum_{\sigma\in S_n}\text{sgn}(\sigma, x_1,\ldots,x_n)\psi x_{\sigma(1)}hx_{\sigma(2)}\ldots h x_{\sigma(n)} \phi
\end{equation}
where $\text{sgn}(\sigma, x_1,\ldots,x_n)=\pm 1$ is chosen such that $\text{sgn}(\id, x_1,\ldots,x_n)=1$ and $\Phi_n$ satisfies \eqref{eq: L infty morphism antisymmetric}.
It is not difficult to check that $\Phi_1$ is a quasi-isomorphism and the $\Phi_n$'s satisfy \eqref{eq: L infty morphism}.
\end{proof}

\subsection{Coherent sheaves and an equivalent of categories}\label{subsection: cohesive and coherent}
 Cohesive modules are closely related to coherent sheaves on $X$. Let $\mathcal{O}_X$ be the sheaf of holomorphic functions. We call a sheaf of $\mathcal{O}_X$-modules $\mathfrak{F}$ \emph{coherent} if it satisfies the following two conditions
\begin{enumerate}
\item $\mathfrak{F}$ is of finite type over $\mathcal{O}_X$, that is, every point in $X$ has an open neighborhood $U$ in $X$ such that there is a surjective morphism $\mathcal {O}_{X}^{n}|_{U}\twoheadrightarrow \mathcal {F}|_{U}$ for some natural number $n$;
\item for \emph{any} open set $U\subseteq X$, \emph{any} natural number $n$, and \emph{any} morphism $\varphi :\mathcal {O}_{X}^{n}|_{U}\to \mathcal {F}|_{U}$ of $\mathcal {O}_{X}$-modules, the kernel of $\varphi$  is of finite type.
\end{enumerate}

Let $D^b_{\coh}(X)$ be the derived category of bounded complexes of $\mathcal{O}_X$-modules with coherent cohomologies.

\begin{thm}\label{thm: equiv of cats}[\cite[Theorem 4.3]{block2010duality}, \cite[Theorem 6.5.1]{bismut2021coherent}]
If $X$ is a compact complex manifold, then  there exists an equivalence $\underline{F}_X: \underline{B}(X)\overset{\sim}{\to}D^b_{\coh}(X)$ as triangulated categories. Here $\underline{B}(X)$ is the homotopy category of $B(X)$.
\end{thm}

In \cite{chuang2021maurer} the result of Theorem \ref{thm: equiv of cats} is generalized to noncompact complex manifold. Recall that a coherent sheaf $\mathfrak{F}$ is called \emph{globally bounded} if  there exists an open covering $U_i$ of $X$ and  integers $a<b$ and $N>0$ such that on each $U_i$ there exists a bounded complex of finitely generated locally free $\mathcal{O}_X$-modules $\mathcal{S}^{\bullet}_i$ which is concentrated in degrees $[a, b]$ and each $\mathcal{S}^j_i$ has rank $\leq N$, together with  a quasi-isomorphism $\mathcal{S}^{\bullet}_i\to \mathfrak{F}^{\bullet}|_{U_i}$.

Let $D^{\gb}_{\coh}(X)$ be the full subcategory of $D^b_{\coh}(X)$ whose objects are bounded complexes of $\mathcal{O}_X$-modules with globally bounded coherent cohomologies. When $X$ is compact, it is clear that $D^{\gb}_{\coh}(X)$ coincides with $D^b_{\coh}(X)$.

\begin{thm}\label{thm: equiv of cats noncompact}[\cite[Theorem 8.3]{chuang2021maurer}]
If $X$ is a  complex manifold, then there exists an equivalence  $\underline{F}_X: \underline{B}(X)\overset{\sim}{\to}D^{\gb}_{\coh}(X)$ as triangulated categories.
\end{thm}

\begin{rmk}
In \cite[Theorem 8.3]{chuang2021maurer}, the result is stated for the derived category of globally bounded perfect complexes instead of $D^{\gb}_{\coh}(X)$. Nevertheless it is easy to see that these two categories are equivalent for nonsingular $X$.
\end{rmk}

For later applications we want to explicitly state the  following results, which are implied in Theorem \ref{thm: equiv of cats} and \ref{thm: equiv of cats noncompact}.

\begin{coro}\label{coro: quasi-isom is homotopy equivalence}
Any quasi-isomorphism in $D^{\gb}_{\coh}(X)$ is induced by a homotopy equivalence in $B(X)$.
\end{coro}

\begin{coro}\label{coro: Extension and morphism in B(X)}
For $\mathcal{S}\in D^{\gb}_{\coh}(X)$ and  $\mathcal{E}\in B(X)$ such that $\underline{F}_X(\mathcal{E})\simeq \mathcal{S}$, we have
\begin{equation}
\Hom_{D^{\gb}_{\coh}(X)}(\mathcal{S},\mathcal{S}[i])\cong \Hom_{\underline{B}(X)}(\mathcal{E},\mathcal{E}[i])\text{, for any } i.
\end{equation}
In particular if $\mathcal{S}$ is a single globally bounded coherent sheaf, then
\begin{equation}
\Ext^i_X(\mathcal{S},\mathcal{S})\cong \Hom_{\underline{B}(X)}(\mathcal{E},\mathcal{E}[i])\text{, for any } i\geq 0.
\end{equation}
\end{coro}

\subsection{Pull-backs of cohesive modules}\label{subsection: pull backs}
Let $f: X\to Y$ be a holomorphic map between complex manifolds.

\begin{lemma}\label{lemma: pull-back of coherent is coherent}
 Let  $\mathcal{E}$ be a bounded complexes of $\mathcal{O}_Y$-modules with globally bounded coherent cohomologies. Then
\begin{equation}
f^*\mathcal{E}:=f^{-1}\mathcal{E}\otimes_{f^{-1}\mathcal{O}_Y}\mathcal{O}_X
\end{equation}
is a bounded complexes of $\mathcal{O}_X$-modules with globally bounded coherent cohomologies. 
\end{lemma}
\begin{proof}
The coherence is given by \cite[Section 1.2.6]{grauert1984coherent}. The global boundedness is clear from the definition and the fact that $f^*\mathcal{O}_Y^N=\mathcal{O}_X^N$.
\end{proof}

Hence we can define
the left derived functor 
\begin{equation}\label{eq: left derived pull-back functor}
Lf^*: D^{\gb}_{\coh}(Y)\to D^{\gb}_{\coh}(X).
\end{equation}

\begin{lemma}\label{lemma: derived pull-back flat modules}
If $\mathcal{E}\in D^{\gb}_{\coh}(Y)$ is a bounded complex of flat $\mathcal{O}_Y$-modules, then we have
\begin{equation}\label{eq: derived pull-back of flat modules}
Lf^*\mathcal{E}=f^*\mathcal{E}.
\end{equation}
\end{lemma}
\begin{proof}
By \cite[\href{https://stacks.math.columbia.edu/tag/064K}{Tag 064K}]{stacks-project}, any bounded complex of flat modules is K-flat. Then the lemma is a consequence of \cite[\href{https://stacks.math.columbia.edu/tag/06YJ}{Tag 06YJ}]{stacks-project}.
\end{proof}

We can also define the pull-backs of cohesive modules. Notice that $f^*$ maps $T^{*}Y$ to $T^{*}X$, hence $\wedge\overline{T^{*}X}$ is a $\wedge \overline{f^*T^{*}Y}$-module.

\begin{defi}\label{defi: pull-back of cohesive modules}
Let $\mathcal{E}=(E^{\bullet}, A^{E^{\bullet}\prime\prime})\in B(Y)$ be a cohesive module on $Y$. We define its pull-back $f^*_b\mathcal{E}$ to be 
$$
(f^*E^{\bullet}, f^*A^{E^{\bullet}\prime\prime})
$$
where $f^*E^{\bullet}$ is the pull-back graded vector bundle and $f^*A^{E^{\bullet}\prime\prime}$ is the pull-back superconnection. In more details, if
$$
A^{E^{\bullet}\prime\prime}=v_0+\nabla^{E^{\bullet}\prime\prime}+v_2+\ldots
$$ 
is the decomposition in \eqref{eq: decomposition of anti super conn}. Then 
\begin{equation}\label{eq: decomposition of the pull-back connection}
f^*A^{E^{\bullet}\prime\prime}=f^*v_0+f^*\nabla^{E^{\bullet}\prime\prime}+f^*v_2+\ldots
\end{equation}
where $f^*\nabla^{E^{\bullet}\prime\prime}$ is the pull-back connection on $f^*E^{\bullet}$, and $f^*v_i$ is the pull-back form valued in $\wedge^{i}\overline{T^{*}X}  \hat{\otimes}  \End^{1-i}(E^{\bullet})$.

If $\phi: \mathcal{E}\to \mathcal{F}$ is a morphism, then we have the pull-back morphism $f^*_b\phi: f^*_b\mathcal{E}\to f^*_b\mathcal{F}$ defined by pulling back each component of $\phi$.
\end{defi}

In particular, if $i: X\hookrightarrow Y$ is a closed embedding, then we denote $i^*_b\mathcal{E}$ by $\mathcal{E}|_X$.

It is easy to see that $f^*_b$ defines a dg-functor $B(Y)\to B(X)$ hence we get the functor $\underline{f^*_b}: \underline{B}(Y)\to \underline{B}(X)$. Moreover, we have the following result.

\begin{prop}\label{prop: pull-back is the derived pull-back}
Under the equivalence of categories in Theorem \ref{thm: equiv of cats} and Theorem \ref{thm: equiv of cats noncompact}, $\underline{f^*_b}: \underline{B}(Y)\to \underline{B}(X)$ is compatible with the left derived pull-back functor $Lf^*: D^{\gb}_{\coh}(Y)\to D^{\gb}_{\coh}(X)$.
\end{prop}
\begin{proof}
The proof is the same as that of \cite[Proposition 6.6]{bismut2021coherent}: We can check that for any $\mathcal{E}\in \underline{B}(Y)$, its image $\underline{F}_Y(\mathcal{E})\in  D^{\gb}_{\coh}(Y)$ is a bounded complex of flat $\mathcal{O}_Y$-modules. Then the proposition is a consequence of Lemma \ref{lemma: derived pull-back flat modules} and Definition \ref{defi: pull-back of cohesive modules}. Notice that we do not need $X$ or $Y$ to be compact.
\end{proof}

\section{Deformations of cohesive modules}\label{section: deformations in general}
\subsection{Deformations and regular deformations}\label{subsection: deformations and regular deformations}
Let $X$ be a compact complex manifold . Let $\Delta\subset \mathbb{C}^{m}$ be a small ball centered at $0$. Let $\mathcal{O}(\Delta)$ be the ring of holomorphic functions on $\Delta$ and $\mathcal{M}(\Delta)$ be the ideal of $\mathcal{O}(\Delta)$ consisting of holomorphic functions vanishing at $0\in \Delta$.

\begin{defi}\label{defi: deformation of cohesive modules}
For a cohesive module $\mathcal{E}\in B(X)$, a family of deformation of $\mathcal{E}$ over $\Delta$ is a cohesive module $\mathfrak{F}$
on $X\times \Delta$ together with a homotopy equivalence 
\begin{equation}\label{eq: deformation restrict to 0}
\phi: \mathfrak{F}|_{X\times \{0\}}\overset{\sim}{\to} \mathcal{E}.
\end{equation}

Two families of deformations $(\mathfrak{F},\phi)$ and $(\mathfrak{G},\psi)$ of  $\mathcal{E}$ over $\Delta$ are  called equivalent if there exists a homotopy equivalence $\theta: \mathfrak{F}\to \mathfrak{G}$ such that
\begin{equation}
\psi\circ \theta|_{X\times \{0\}}=\phi \text{ up to chain homotopy.}
\end{equation}
We denote  the set of equivalent classes of  deformations of $\mathcal{E}$  over $\Delta$ by $\text{Def}_{\Delta}(\mathcal{E})$.
\end{defi}

\begin{rmk}
In this paper we usually consider \emph{germs} of families of deformations, which means we will choose $\Delta$ sufficiently small.
\end{rmk}

\begin{lemma}\label{lemma: homotopy equivalence and deformations}
$\text{Def}_{\Delta}(-)$ give a well-defined functor from $D^b_{\coh}(X)$ to Sets.
\end{lemma}
\begin{proof}
By Corollary \ref{coro: quasi-isom is homotopy equivalence}, homotopy equivalences in $B(X)$ correspond to  quasi-isomorphisms in $D^b_{\coh}(X)$.  If $\mathcal{E}_1$ and $\mathcal{E}_2$ are homotopy equivalent in $B(X)$. Then by definition  there is a bijection between $\text{Def}_{\Delta}(\mathcal{E}_1)$ and $\text{Def}_{\Delta}(\mathcal{E}_2)$ given by compositions. 
\end{proof}

Let $p: X\times \Delta \to X$ and $q: X\times \Delta\to \Delta$ be the projections. 
A family of deformation $\mathfrak{F}$ of $\mathcal{E}$ consists of a bounded, graded $C^{\infty}$-vector bundle $F^{\bullet}$ on $X\times \Delta$ and a flat superconnection $A^{F^{\bullet}\prime\prime}$. 
Since $X$ is compact, by Ehresmann's fibration theorem, for $\Delta$ sufficiently small, $F^{\bullet}$  is $C^{\infty}$-isomorphic to $p^*(F^{\bullet}|_{X\times \{0\}})$. Therefore we can consider  $\mathfrak{F}=(F^{\bullet}, A^{F^{\bullet}\prime\prime})$ with underline graded vector bundle constant along the $\Delta$ direction while its superconnection $A^{F^{\bullet}\prime\prime}$ varies. 

In general, in the viewpoint of the decomposition \eqref{eq: decomposition of anti super conn}, $A^{F^{\bullet}\prime\prime}$ has higher degree component in both $X$ and $\Delta$ directions. In more details, we use the notation
\begin{equation}\label{eq: wedge s t}
\wedge^{i,j}\overline{T^{*}(X\times \Delta)}:=p^*\wedge^{i}\overline{T^{*}X}\times q^*\wedge^{j}\overline{T^{*}\Delta}.
\end{equation}
Hence 
\begin{equation}\label{eq: decomposition X and Delta}
\wedge^{l}\overline{T^{*}(X\times \Delta)}=\bigoplus_{i+j=l}\wedge^{i,j}\overline{T^{*}(X\times \Delta)}.
\end{equation}
Notice that $\wedge^{i,j}\overline{T^{*}(X\times \Delta)}$ is not the bundle of $(i,j)$-forms in the usual sense.

Since $F^{\bullet}$  is $C^{\infty}$-isomorphic to $p^*(F^{\bullet}|_{X\times \{0\}})$, we have the  flat $\dbar$-connection in the $\Delta$ direction 
\begin{equation}\label{eq: dbar Delta}
\dbar_{\Delta}: \wedge^{i,j}\overline{T^{*}(X\times \Delta)}\times F^{\bullet}\to  \wedge^{i,j+1}\overline{T^{*}(X\times \Delta)}\times F^{\bullet}.
\end{equation}

by \eqref{eq: decomposition of anti super conn}, we have the decomposition
\begin{equation}\label{eq: decompose AF}
A^{F^{\bullet}\prime\prime}=\mathfrak{u}_0+\nabla^{F^{\bullet}\prime\prime}+\mathfrak{u}_2+\ldots
\end{equation}
where 
$$
\nabla^{F^{\bullet}\prime\prime}: F^{\bullet}\to \overline{T^{*}(X\times \Delta)}  \times F^{\bullet}
$$
 is a  $\dbar_{X\times \Delta}$-connection, and
$$
\mathfrak{u}_k\in C^{\infty}(X\times \Delta,\wedge^{k}\overline{T^{*}(X\times \Delta)}  \hat{\otimes}  \End^{1-k}(F^{\bullet}))
$$
is  $C^{\infty}(X\times \Delta)$-linear.  

Using \eqref{eq: decomposition X and Delta} we can further decompose $\nabla^{F^{\bullet}\prime\prime}$ as
$
\nabla^{F^{\bullet}\prime\prime}=\nabla^{F^{\bullet}\prime\prime}_X+\nabla^{F^{\bullet}\prime\prime}_{\Delta}
$
where
\begin{equation}
\nabla^{F^{\bullet}\prime\prime}_X: F^{\bullet}\to \wedge^{1,0}\overline{T^{*}(X\times \Delta)}  \times F^{\bullet}
\end{equation}
and
\begin{equation}
\nabla^{F^{\bullet}\prime\prime}_{\Delta}: F^{\bullet}\to \wedge^{0,1}\overline{T^{*}(X\times \Delta)}  \times F^{\bullet}.
\end{equation}
Also we decompose $
\mathfrak{u}_k=\sum_{i+j=k}\mathfrak{u}_{i,j}$
where
\begin{equation}
\mathfrak{u}_{i,j}\in C^{\infty}(X\times \Delta,\wedge^{i,j}\overline{T^{*}(X\times \Delta)}  \hat{\otimes}  \End^{1-i-j}(F^{\bullet}))
\end{equation}

We can regroup the sum in \eqref{eq: decompose AF} as follows. Let
\begin{equation}
\chi_0:=\nabla^{F^{\bullet}\prime\prime}_X+\sum_{i\geq0, i\neq 1} \mathfrak{u}_{i,0}
\end{equation}
and for $j\geq 2$
\begin{equation}
\chi_j=\sum_{i\geq 0}\mathfrak{u}_{i,j}
\end{equation}
We further write $\chi_1=\nabla^{F^{\bullet}\prime\prime}_{\Delta}-\dbar_{\Delta}$.
Notice that $\chi_0$ is a $\dbar_X$-superconnection on $F^{\bullet}|_{X\times \{0\}}$; and for each $i\geq 1$,
\begin{equation}
\chi_i \in \bigoplus_{k\geq 0}C^{\infty}(X\times \Delta,\wedge^{k,i}\overline{T^{*}(X\times \Delta)}  \hat{\otimes}  \End^{1-i-k}(F^{\bullet})).
\end{equation}

Then we can rewrite  \eqref{eq: decompose AF} as 
\begin{equation}\label{eq: decompose AF in chi}
A^{F^{\bullet}\prime\prime}=\chi_0+\dbar_{\Delta}+\chi_1+\chi_2+\ldots
\end{equation}

\begin{defi}\label{defi: regular deformation}
A deformation $\mathfrak{F}=(F^{\bullet}, A^{F^{\bullet}\prime\prime})$ of $\mathcal{E}$ is called a \emph{regular deformation} if in the decomposition \eqref{eq: decompose AF in chi} we have
$\chi_i=0$ for each $i\geq 1$.
\end{defi}

For a regular deformation, we know that
\begin{equation}
(\chi_0+\dbar_{\Delta})^2=0.
\end{equation}
Therefore we know that
$\chi_0$ is a flat $\dbar$-superconnection in the $X$ direction such that
\begin{equation}\label{eq: superconnection is holomorphic}
[\chi_0,\dbar_{\Delta}]=0.
\end{equation}
Here $[\chi_0,\dbar_{\Delta}]=\chi_0\dbar_{\Delta}+\dbar_{\Delta}\chi_0$ is the supercommutator.
Intuitively, a regular deformation is a holomorphic family of objects in $B(X)$ over $\Delta$.

\begin{prop}\label{prop: deformation is equivalent to a regular one}
 Any deformation $\mathfrak{F}=(F^{\bullet}, A^{F^{\bullet}\prime\prime})$ is equivalent to a regular deformation if $\Delta$ is sufficiently small.
\end{prop}
\begin{proof}
We can prove that on sufficiently small $\Delta$, there exists an invertible degree zero morphism $J\in \Hom_{B(X\time \Delta)}(\mathfrak{F},\mathfrak{F})$ ($J$ is in general not closed in $ \Hom_{B(X\time \Delta)}(\mathfrak{F},\mathfrak{F})$), such that
\begin{equation}\label{eq: conjugate}
J^{-1}\circ A^{F^{\bullet}\prime\prime}\circ J=\widetilde{\chi_0}+\dbar_{\Delta},
\end{equation}

The proof is similar to that of \cite[Theorem 5.2.1]{bismut2021coherent}. We denote
$\chi_0+\chi_1+\chi_2+\ldots$ by $B_0$ and $\chi_1+\chi_2+\ldots$ by $B_0^{\geq 1}$. Then we know that
\begin{equation}
B_0^{\geq 1}\in \bigoplus_{i\geq 0, j\geq 1}C^{\infty}(X\times \Delta,\wedge^{i,j}\overline{T^{*}(X\times \Delta)}  \hat{\otimes}  \End^{1-i-j}(F^{\bullet})).
\end{equation}
and
\begin{equation}
A^{F^{\bullet}\prime\prime}=\dbar_{\Delta}+B_0=\chi_0+\dbar_{\Delta}+B_0^{\geq 1}.
\end{equation}

We consider
\begin{equation}
\mathfrak{D}^{j,k}:=C^{\infty}(X, \wedge^{j}\overline{T^{*}X}  \hat{\otimes}  \End^{k}(F^{\bullet}))
\end{equation}
as an infinite dimensional vector bundle on $\Delta$. Therefore we have
\begin{equation}
B_0^{\geq 1}\in \bigoplus_{i\geq 0, j\geq 1}C^{\infty}(\Delta,\wedge^{j}\overline{T^{*}\Delta} \hat{\otimes} \mathfrak{D}^{i,1-i-j}).
\end{equation}

Since $X$ is compact, we can impose a norm on $\mathfrak{D}^{\bullet,\bullet}$ and shrink $\Delta$ if necessary so that $B_0^{\geq 1}$ is a bounded section on $\Delta$. Then we proceed the same as in the  proof of \cite[Theorem 5.2.1]{bismut2021coherent} to get a
\begin{equation}
J\in \bigoplus_{i,j\geq 0}C^{\infty}(X\times \Delta,\wedge^{i,j}\overline{T^{*}(X\times \Delta)}  \hat{\otimes}  \End^{-i-j}(F^{\bullet})).
\end{equation}
such that 
\begin{enumerate}
\item $J$ is invertible;
\item  the conjugation by $J$ eliminates $B_0^{\geq 1}$.
\end{enumerate}
More precisely, we have
\begin{equation}
J^{-1}\circ A^{F^{\bullet}\prime\prime}\circ J=J^{-1}\circ(\chi_0+\dbar_{\Delta}+B_0^{\geq 1})\circ J=\widetilde{\chi_0}+\dbar_{\Delta},
\end{equation}
where $\widetilde{\chi_0}$ is again a $\dbar_X$-superconnection on $F^{\bullet}|_{X\times \{0\}}$. 

Let $\widetilde{\mathfrak{F}}=(F^{\bullet},\widetilde{\chi_0}+\dbar_{\Delta})$. Then  \eqref{eq: conjugate}
 tells us that 
$$
(\widetilde{\chi_0}+\dbar_{\Delta})^2=0
$$
hence $\widetilde{\mathfrak{F}}\in B(X\times \Delta)$, and moreover
\begin{equation}
A^{F^{\bullet}\prime\prime}\circ J=J\circ (\widetilde{\chi_0}+\dbar_{\Delta}),
\end{equation}
i.e. $J$ is a degree zero closed morphism from $\widetilde{\mathfrak{F}}$ to $\mathfrak{F}$.  $J$ is actually a homotopy equivalence since $J$ is invertible.

 Let 
\begin{equation}
\tilde{\phi}=\phi\circ J|_{X\times\{0\}}: \widetilde{\mathfrak{F}}|_{X\times\{0\}}\to \mathcal{E}.
\end{equation} 
Then it is clear that $\tilde{\phi}$ is a homotopy equivalence. We get  that $( \widetilde{\mathfrak{F}}, \tilde{\phi})$ is a regular deformation of $\mathcal{E}$ on $\Delta$ and $J: ( \widetilde{\mathfrak{F}}, \tilde{\phi})\to (\mathfrak{F},\phi)$ is an equivalence. 
\end{proof}

\begin{rmk}
In the proof of Proposition \ref{prop: deformation is equivalent to a regular one} we require $X$ to be compact. This and the Ehresmann's fibration theorem are the two main reasons that we require $X$ to be compact in this paper.
\end{rmk}

We will denote $\widetilde{\chi_0}$ in \eqref{eq: conjugate} by $A^{F^{\bullet}\prime\prime}_t$ in order to emphasize that it is a superconnection on $F^{\bullet}$ over $X$, depending on the parameter $t\in \Delta$. Equation \eqref{eq: superconnection is holomorphic} indicates that $A^{F^{\bullet}\prime\prime}_t$ is holomorphic  with respect to $t$.

We can decompose 
\begin{equation}
A^{F^{\bullet}\prime\prime}_t=u_0(t)+\nabla^{F^{\bullet}\prime\prime}_t+u_2(t)+\ldots
\end{equation}
as before. It is clear that each $u_i(t)$ and $\nabla^{F^{\bullet}\prime\prime}_t$ are holomorphic with respect to $t$. In particular, $A^{F^{\bullet}\prime\prime}_0$ is the superconnection of the cohesive module $\mathfrak{F}|_{X\times \{0\}}\in B(X)$.

To simplify the notation we denote $\mathfrak{F}|_{X\times \{0\}}$ by $\mathfrak{F}_0$.
Recall that we have the DGLA  $L_{\mathfrak{F}_0}$ as given in Section \ref{subsection: DGLA LE}, and we consider the element
\begin{equation}
\eta(t):=A^{F^{\bullet}\prime\prime}_t-A^{F^{\bullet}\prime\prime}_0.
\end{equation}

\begin{prop}\label{prop: Maurer-Cartan in F}
The element $\eta(t)\in L^1_{\mathfrak{F}_0}\otimes \mathcal{O}(\Delta)$ satisfies
\begin{equation}\label{eq: Maurer-Cartan for F}
\diff_{\mathfrak{F}}\eta(t)+\frac{1}{2}[\eta(t),\eta(t)]=0
\end{equation}
and $\eta(0)=0$. In other words, $\eta(t)$ is a Maurer-Cartan element in $L^1_{\mathfrak{F}_0}\otimes \mathcal{M}(\Delta)$.
\end{prop}

\begin{rmk}\label{rmk: the topological tensor product}
Here $L^1_{\mathfrak{F}_0}\otimes \mathcal{O}(\Delta)$ and $L^1_{\mathfrak{F}_0}\otimes \mathcal{M}(\Delta)$ are completed tensor products, i.e. an element  in $L^1_{\mathfrak{F}_0}\otimes \mathcal{O}(\Delta)$ is an element $\eta(t)\in \bigoplus_{i+j=1}C^{\infty}(X\times \Delta,\wedge^{i}\overline{T^{*}X}\hat{\otimes} \End^{j}(F^{\bullet}))$ which is holomorphic in the $\Delta$ direction. Moreover $\eta(t)\in L^1_{\mathfrak{F}_0}\otimes \mathcal{O}(\Delta)$ means that in addition $\eta(0)\equiv 0$.
\end{rmk}

\begin{proof}[Proof of Proposition \ref{prop: Maurer-Cartan in F}]
It is clear that $\eta(t)=A^{F^{\bullet}\prime\prime}_t-A^{F^{\bullet}\prime\prime}_0$ belongs to  $L^1_{\mathfrak{F}_0}$ for each $t$, and $\eta(0)=0$. Since 
$A^{F^{\bullet}\prime\prime}_t$ is holomorphic  with respect to $t$, $\eta(t)$ is also holomorphic with respect to $t$. Hence $\eta(t)\in L^1_{\mathfrak{F}_0}\otimes \mathcal{O}(\Delta)$.

Since $A^{F^{\bullet}\prime\prime}_t=A^{F^{\bullet}\prime\prime}_0+\eta(t)$, Equation \eqref{eq: Maurer-Cartan for F} follows from the fact that $(A^{F^{\bullet}\prime\prime}_t)^2=0$. 
\end{proof}

Recall that for a DGLA $\mathfrak{g}$ and a $\mathbb{C}$-algebra $\mathcal{R}$ with maximal ideal $m$, two Maurer-Cartan elements $\alpha_1$ and $\alpha_2$ in  $\mathfrak{g}^1\otimes m$ are called gauge equivalent if there exists a $u\in\mathfrak{g}^0\otimes m$ such that
\begin{equation}\label{eq: gauge equivalent}
e^u\circ (\diff_{\mathfrak{g}}+\alpha_1)\circ e^{-u}=\diff_{\mathfrak{g}}+\alpha_2
\end{equation}
given that $e^u$ converges in $\mathfrak{g}^0\otimes \mathcal{R}$. Let MC$(\mathfrak{g}\otimes m)$ denote the set of gauge equivalent classes of Maurer-Cartan elements of $\mathfrak{g}\otimes m$. 

In this paper we take $\mathcal{R}=\mathcal{O}(\Delta)$ and $m=\mathcal{M}(\Delta)$. We have the following result.

\begin{coro}\label{coro: deformation gives Maurer-Cartan element in E}
A deformation of $\mathcal{E}$ over sufficiently small $\Delta$ gives an element in MC$(L_{\mathcal{E}}\otimes \mathcal{M}(\Delta))$.
\end{coro}
\begin{proof}
Proposition \ref{prop: Maurer-Cartan in F} tells us that a deformation $\mathfrak{F}$ gives an element in MC$(L_{\mathfrak{F}_0}\otimes \mathcal{M}(\Delta))$. Since $\phi: \mathfrak{F}_0\to \mathcal{E}$ is a homotopy equivalence, we can choose a homotopy inverse $\psi: \mathcal{E}\to \mathfrak{F}_0$ and a homotopy operator $h\in \Hom_{B(X)}(\mathfrak{F}_0,\mathfrak{F}_0)$ such that 
$$
\psi\phi-\id_{F^{\bullet}}=A^{F^{\bullet}\prime\prime}_0 h+hA^{F^{\bullet}\prime\prime}_0.
$$
Lemma \ref{lemma: homotopy equivalent L infty equivalent of dgla} tells us that from $\phi$, $\psi$, and $h$ we can construct an $L_{\infty}$ quasi-isomorphism
\begin{equation}
\Phi: L_{\mathfrak{F}_0}\otimes \mathcal{M}(\Delta)\to L_{\mathcal{E}}\otimes \mathcal{M}(\Delta).
\end{equation}

Now as in \cite[Section 8]{yekutieli2012mc}, for a Maurer-Cartan element $\eta(t)\in L_{\mathfrak{F}_0}\otimes \mathcal{M}(\Delta)$. We can construct a element
\begin{equation}
\text{MC}(\Phi,\eta(t)):=\sum_{n\geq 1}\frac{1}{n!}\Phi_n(\eta(t),\ldots,\eta(t)).
\end{equation}
Actually for the $\Phi$ given in \eqref{eq: L infty morphism for cohesive module equivalence}, we have
\begin{equation}\label{eq: MC Phi eta explicit}
\text{MC}(\Phi,\eta(t))=\sum_{n\geq 1}\underbrace{\psi \eta(t) h\eta(t)h\ldots h\eta(t)\phi}_{n ~\eta(t)\text{'s}}.
\end{equation}
\cite[Proposition 8.2]{yekutieli2012mc} assures that $\text{MC}(\Phi,\eta(t))$ is a Maurer-Cartan element in $L_{\mathcal{E}}\otimes \mathcal{M}(\Delta)$.
\end{proof}

\begin{rmk}
\cite[Theorem 8.13]{yekutieli2012mc} further tells us that $\text{MC}(\Phi,-)$ gives a bijection from MC$(L_{\mathfrak{F}_0}\otimes \mathcal{M}(\Delta))$ to MC$(L_{\mathcal{E}}\otimes \mathcal{M}(\Delta))$.
\end{rmk}

Although  Corollary \ref{coro: deformation gives Maurer-Cartan element in E} is useful, we still have the following problems:
\begin{enumerate}
\item We do not know the dependence of $\text{MC}(\Phi,-)$ on the choices of $\psi$ and $h$;
\item We do not know the relation between equivalence of deformations and gauge equivalence of Maurer-Cartan elements.
\end{enumerate}

To solve these problems, we need to further investigate equivalences of deformations.

Let $\theta: \mathfrak{F}\to \mathfrak{G}$ be an equivalence between regular deformations of $\mathcal{E}$ over $\Delta$. According to \eqref{eq: decomposition of morphism}, \eqref{eq: wedge s t}, and \eqref{eq: decomposition X and Delta}, $\theta$ decomposes into a finite sum
\begin{equation}\label{eq: decompose theta into two directions}
\theta=\sum_{i,j\geq 0}\theta^{i,j}
\end{equation}
where
\begin{equation}
\theta^{i,j}\in C^{\infty}(X\times\Delta,\wedge^{i,j}\overline{T^{*}(X\times \Delta)}\hat{\otimes} \Hom^{-i-j}(F^{\bullet},G^{\bullet})).
\end{equation}
Moreover let
\begin{equation}\label{eq: theta decompose Delta direction}
\theta^{\bullet, j}:=\sum_{r} \theta^{i,j}.
\end{equation}

\begin{defi}\label{defi: regular equivalence}
An equivalence $\theta: \mathfrak{F}\to \mathfrak{G}$ between regular deformations is called a \emph{regular equivalence} if $\theta^{\bullet, j}=0$ for all $j>0$.
\end{defi}

Conceptually, a regular deformation is a family of cohesive modules on $X$ whose superconnection varies holomorphically with respect to $t\in \Delta$, and a regular equivalence is a family of homotopy equivalences on $X$ which varies holomorphically with respect to $t\in \Delta$.

\begin{prop}\label{prop: equivalence and regular equivalence}
Any equivalence $\theta: \mathfrak{F}\to \mathfrak{G}$ between regular deformations is cochain homotopic to a regular equivalence over a sufficiently small neighborhood $\Delta$ of $0$.
\end{prop}
\begin{proof}
Since both $\mathfrak{F}$ and $\mathfrak{G}$ are regular, we can write
$$
\mathfrak{F}=(F^{\bullet}, A^{F^{\bullet}\prime\prime}_t+\dbar_{\Delta})
$$
and
$$
\mathfrak{G}=(G^{\bullet}, A^{G^{\bullet}\prime\prime}_t+\dbar_{\Delta}).
$$
Since $\theta: \mathfrak{F}\to \mathfrak{G}$ is a closed, degree $0$ morphism, we have
\begin{equation}\label{eq: theta is closed}
(A^{G^{\bullet}\prime\prime}_t+\dbar_{\Delta})\theta=\theta(A^{F^{\bullet}\prime\prime}_t+\dbar_{\Delta}),
\end{equation}
To simplify the notation, let us denote the commutator $A^{G^{\bullet}\prime\prime}_t(\cdot)-(\cdot)A^{F^{\bullet}\prime\prime}_t$ by $D_t(\cdot)$. It is clear that
\begin{equation}\label{eq: dbar and DX commute}
\dbar_{\Delta}D_t+D_t\dbar_{\Delta}=0.
\end{equation}

Under the decomposition in \eqref{eq: theta decompose Delta direction},  Equation \eqref{eq: theta is closed} decomposes into
\begin{equation}\label{eq: DX theta 0}
D_t(\theta^{\bullet, 0})=0
\end{equation}
and
\begin{equation}\label{eq: two direction closeness of theta}
\dbar_{\Delta}(\theta^{\bullet, j})+D_t(\theta^{\bullet, j+1})=0 \text{, for all } j\geq 0.
\end{equation}
Since the sum in \eqref{eq: decompose theta into two directions} is finite, there exists an $N$ such that 
$$
\theta^{\bullet, j}=0 \text{, for all } j>N.
$$
Therefore \eqref{eq: two direction closeness of theta} gives
\begin{equation}
\dbar_{\Delta}(\theta^{\bullet, N})=0.
\end{equation}

For $\Delta$ sufficiently small, the Dolbeault cohomology with respect to $\dbar_{\Delta}$ vanishes on $\Delta$. Hence we can find
 $$\gamma^{N-1}\in\bigoplus_i C^{\infty}(X\times\Delta,\wedge^{i,N-1}\overline{T^{*}(X\times \Delta)}\hat{\otimes} \Hom^{-i-N}(F^{\bullet},G^{\bullet}))
$$ such that
\begin{equation}
\dbar_{\Delta}(\gamma^{N-1})=\theta^{\bullet, N}.
\end{equation}
Along the same lines we can find recursively 
\begin{equation}
\gamma^j\in \bigoplus_i C^{\infty}(X\times\Delta,\wedge^{i,j}\overline{T^{*}(X\times \Delta)}\hat{\otimes} \Hom^{-i-j-1}(F^{\bullet},G^{\bullet})), ~j=N-2,\ldots, 0
\end{equation}
such that
\begin{equation}\label{eq: gamma zigzag}
\dbar_{\Delta}(\gamma^j)+D_t(\gamma^{j+1})=\theta^{\bullet, j+1}, ~j=N-2,\ldots, 0.
\end{equation}
In particular we have
\begin{equation}\label{eq: gamma 0}
\dbar_{\Delta}(\gamma^0)+D_t(\gamma^1)=\theta^{\bullet, 1}.
\end{equation}
Let $\tilde{\theta}:=\theta^{\bullet, 0}-D_t(\gamma^0)$. Equation \eqref{eq: dbar and DX commute}, \eqref{eq: DX theta 0}, \eqref{eq: two direction closeness of theta}, and \eqref{eq: gamma 0} tell us 
\begin{equation}
\dbar_{\Delta}(\tilde{\theta})=0,\text{ and } D_t(\tilde{\theta})=0,
\end{equation}
i.e. $\tilde{\theta}: \mathfrak{F}\to \mathfrak{G}$ is a closed, degree $0$ morphism. 
Moreover, by \eqref{eq: gamma zigzag},  $\gamma:=\sum_{i=0}^{N-1}\gamma^i$ is a cochain homotopy between $\tilde{\theta}$ and $\theta$. Therefore $\tilde{\theta}$ is a homotopy equivalence because $\theta$ is. Hence $\tilde{\theta}$ a regular equivalence between  $\mathfrak{F}$ and $\mathfrak{G}$.
\end{proof}

\subsection{Strong deformations of cohesive modules}\label{subsection: strong deformations}
Recall that we are deforming a cohesive module $\mathcal{E}\in B(X)$. Let $\mathcal{E}=(E^{\bullet}, A^{E^{\bullet}\prime\prime})$. We have the DGLA
$L_{\mathcal{E}}$ as given in Section \ref{subsection: DGLA LE}.

We want to find the relation between equivalences of deformations of  $\mathcal{E}$ and gauge equivalences of Maurer-Cartan elements in $L_{\mathcal{E}}$.  For this purpose we introduce the following concept. Recall that $p: X\times \Delta\to X$ is the projection. We denote the pull-back bundle  $p^*(E^{\bullet})$ still by $E^{\bullet}$.

\begin{defi}\label{defi: strong deformation}
For a cohesive module  $\mathcal{E}=(E^{\bullet}, A^{E^{\bullet}\prime\prime})\in B(X)$, a family of \emph{strong deformations} of $\mathcal{E}$ is a cohesive module on $X\times \Delta$ of the form
\begin{equation}
\mathfrak{E}=(E^{\bullet}, A^{E^{\bullet}\prime\prime}+\epsilon(t)+\dbar_{\Delta}).
\end{equation}
Here   $\dbar_{\Delta}$ is the flat $\dbar$-operator on $E^{\bullet}$ in the $\Delta$ direction,  $\epsilon(t)$ is holomorphic with respect to $\dbar_{\Delta}$ and $\epsilon(0)=0$.

We require that the morphism $\phi: \mathfrak{E}|_{X\times \{0\}}\to \mathcal{E}$ to be the identity morphism $\id_{E^{\bullet}}$.
\end{defi}

\begin{rmk}
It is clear that a strong deformation is regular as in Definition \ref{defi: regular deformation}.
\end{rmk}

\begin{prop}\label{prop: Maurer-Cartan in E}
The element $\epsilon(t)\in L^1_{\mathcal{E}}\otimes \mathcal{O}(\Delta)$ satisfies
\begin{equation}\label{eq: Maurer-Cartan for E}
\diff_{\mathcal{E}}\epsilon(t)+\frac{1}{2}[\epsilon(t),\epsilon(t)]=0
\end{equation}
and $\epsilon(0)=0$. In other words, $\epsilon(t)$ is a Maurer-Cartan element in $L^1_{\mathcal{E}}\otimes\mathcal{M}(\Delta)$. Here $L^1_{\mathcal{E}}\otimes \mathcal{O}(\Delta)$ and $L^1_{\mathcal{E}}\otimes \mathcal{O}(\Delta)$ are completed tensor products as in  Remark \ref{rmk: the topological tensor product}.
\end{prop}
\begin{proof}
Same as Proposition \ref{prop: Maurer-Cartan in F}.
\end{proof}

The following proposition shows the relation between strong deformation and general deformation, which is a counterpart of Corollary \ref{coro: deformation gives Maurer-Cartan element in E} on the level of deformations instead of Maurer-Cartan elements.

\begin{prop}\label{prop: deformation and strong deformation}
Let $\mathcal{E}$ be a cohesive module on a compact complex manifold $X$. Any deformation $\mathfrak{F}$ of $\mathcal{E}$ on $\Delta$ is equivalent to a strong deformation of of $\mathcal{E}$ if $\Delta$ is sufficiently small.
\end{prop}
\begin{proof}
It is mostly the \emph{homological perturbation lemma} as in \cite[Theorem 2.3]{crainic2004perturbation}. We give the proof here for completeness. Let $\mathfrak{F}=(F^{\bullet}, A^{F^{\bullet}\prime\prime})$ be a deformation of  $\mathcal{E}$ with homotopy equivalence $\phi: \mathfrak{F}|_{X\times \{0\}}\overset{\sim}{\to} \mathcal{E}$. Without loss of generality, we can assume that $\mathfrak{F}$ is regular as in Definition \ref{defi: regular deformation}, i.e.
$$
A^{F^{\bullet}\prime\prime}=A^{F^{\bullet}\prime\prime}_0+\eta(t)+\dbar_{\Delta},
$$
where $\eta(t)$ is a Maurer-Cartan element in $L^1_{\mathfrak{F}_0}\otimes\mathcal{M}(\Delta)$.  

As before, there exists a homotopy inverse
$$
\psi: \mathcal{E}\overset{\sim}{\to} \mathfrak{F}_0
$$
together with a homotopy operator $h\in \Hom_{B(X)}(\mathfrak{F}_0,\mathfrak{F}_0)$ such that 
$$
\psi\phi-\id_{F^{\bullet}}=A^{F^{\bullet}\prime\prime}_0h+h A^{F^{\bullet}\prime\prime}_0.
$$
We consider
$$
\id_{F^{\bullet}}-\eta(t)h\in  \bigoplus_{i+j=0}C^{\infty}(X\times \Delta,\wedge^{i}\overline{T^{*}X}\hat{\otimes} \End^{j}(F^{\bullet})).
$$
We know $\eta(0)h=0$.  Moreover, since $F^{\bullet}$ is a bounded complex, $\eta(t)h$ has finitely many non-zero components. Therefore for a sufficiently small neighborhood $\Delta$ of $0$, we know $\id_{F^{\bullet}}-\eta(t)h$ is 
invertible with inverse given by
\begin{equation}\label{eq: inverse of eta h}
\id_{F^{\bullet}}+\eta(t)h+(\eta(t)h)^2+\ldots
\end{equation}
which converges to an element in $\bigoplus_{i+j=0}C^{\infty}(X\times \Delta,\wedge^{i}\overline{T^{*}X}\hat{\otimes} \End^{j}(F^{\bullet}))$. 

Now we can construct a cohesive module $(E^{\bullet}, A^{E^{\bullet}\prime\prime}+\epsilon(t)+\dbar_{\Delta})$ with
\begin{equation}\label{eq: perturbed epsilon}
\epsilon(t):=\phi(\id_{F^{\bullet}}-\eta(t)h)^{-1}\eta(t)\psi\in  L^1_{\mathcal{E}}\otimes \mathcal{O}(\Delta).
\end{equation}

It is not difficult to check that $\epsilon(t)$ is a Maurer-Cartan element in $ L^1_{\mathcal{E}}\otimes \mathcal{M}(\Delta)$. See \cite[The proof of Theorem 2.3]{crainic2004perturbation}.

Moreover, we define a degree $0$ morphism 
$$
\phi(t): (F^{\bullet}, A^{F^{\bullet}\prime\prime}_0+\eta(t)+\dbar_{\Delta})\to (E^{\bullet}, A^{E^{\bullet}\prime\prime}+\epsilon(t)+\dbar_{\Delta})
$$
 over $X\times \Delta$ by
\begin{equation}\label{eq: phi(t)}
\phi(t):=\phi+\phi(\id_{F^{\bullet}}-\eta(t)h)^{-1}\eta(t)h.
\end{equation}
It is clear that $\phi(t)$ is holomorphic with respect to $t$. Again in the same way as \cite[The proof of Theorem 2.3]{crainic2004perturbation}, we can easily check that $\phi(t)$ is a closed morphism, i.e. 
\begin{equation}
\phi(t)(A^{F^{\bullet}\prime\prime}_0+\eta(t))=A^{E^{\bullet}\prime\prime}+\epsilon(t)\phi(t).
\end{equation}

Now we need to show that $\phi(t)$ is a homotopy equivalence for sufficiently small $\Delta$. Let
$$
A^{E^{\bullet}\prime\prime}=v_0+\nabla^{E^{\bullet}\prime\prime}+v_2+\ldots
$$
$$
A^{F^{\bullet}\prime\prime}_0=u_0+\nabla^{F^{\bullet}\prime\prime}+u_2+\ldots
$$
be the decompositions in \eqref{eq: decomposition of anti super conn} and let 
$$
A^{E^{\bullet}\prime\prime}+\epsilon(t)=v_0(t)+\nabla^{E^{\bullet}\prime\prime}_t+v_2(t)+\ldots
$$
$$
A^{F^{\bullet}\prime\prime}_0+\eta(t)=u_0(t)+\nabla^{F^{\bullet}\prime\prime}_t+u_2(t)+\ldots
$$
be the similar decompositions. It is clear that $u_0(0)=u_0$ and $v_0(0)=v_0$. Since $\phi: (F^{\bullet}, A^{F^{\bullet}\prime\prime}_0)\overset{\sim}{\to} (E^{\bullet}, A^{E^{\bullet}\prime\prime})$ is a homotopy equivalence, by Proposition \ref{prop: homotopy equiv degree 0}, its degree $0$ component
$$
\phi_0: (F^{\bullet}, u_0)\to  (E^{\bullet}, v_0)
$$
is a quasi-isomorphism of cochain complexes. Therefore for $\Delta$ small enough, the degree $0$ component of $\phi(t)$,
$$
\phi_0(t): (F^{\bullet}, u_0(t))\to  (E^{\bullet}, v_0(t))
$$
is a quasi-isomorphism for any $t\in \Delta$. Again by Proposition \ref{prop: homotopy equiv degree 0}, $\phi(t)$ is a homotopy equivalence.
\end{proof}

\begin{rmk}
The $\epsilon(t)$ constructed in \eqref{eq: perturbed epsilon} is the same as the one in \eqref{eq: MC Phi eta explicit} in the proof of  Corollary \ref{coro: deformation gives Maurer-Cartan element in E}.
\end{rmk}

\begin{rmk}
If $\mathfrak{F}$ is a regular deformation of $\mathcal{E}$, then the equivalence $\phi(t)$ constructed in the proof of Proposition \ref{prop: deformation and strong deformation} is a regular equivalence as in Definition \ref{defi: regular equivalence}.
\end{rmk}

It is clear that a holomorphic vector bundle $E$ can be considered as an object in $B(X)$. We have the following corollary.

\begin{coro}\label{coro: deformation of holomorphic vector bundles}
Any deformation as cohesive modules of a holomorphic vector bundle $E$ over sufficiently small $\Delta$ is equivalent to  a holomorphic vector bundle.
\end{coro}
\begin{proof}
By Proposition \ref{prop: deformation and strong deformation}, any deformation of $E$ is equivalent of a strong deformation of $E$ . When $E$ is a holomorphic vector bundle on $X$, a strong  deformation of $E$ is a holomorphic vector bundle on $X\times \Delta$.
\end{proof}

It is clear that the strong deformation constructed in the proof of Proposition \ref{prop: deformation and strong deformation} depends on the choice of the homotopy inverse $\psi$ and the cochain homotopy operator $h$. 
To further investigate equivalences between strong deformations, we introduce the following concept.

\begin{defi}\label{defi: strong equivalence}
An equivalence $\theta: \mathfrak{E}_1\to \mathfrak{E}_2$ between two strong deformations of $\mathcal{E}$ over $\Delta$ is called a \emph{strong equivalence} if $\theta$ is a regular equivalence and moreover
\begin{equation}
\theta|_{X\times \{0\}}=\id_{E^{\bullet}}.
\end{equation}
In other words, $\theta$ is of the form
\begin{equation}
\theta=\id_{E^{\bullet}} +\vartheta(t), ~t\in \Delta
\end{equation}
where $\vartheta(t)$ is holomorphic with respect to $t$ and $\vartheta(0)=0$.
\end{defi}

\begin{prop}\label{prop: equivalence and strong equivalence}
Any equivalence $\theta: \mathfrak{E}_1\to \mathfrak{E}_2$ between strong deformations is cochain homotopic to a strong equivalence over a sufficiently small neighborhood $\Delta$ of $0$.
\end{prop}
\begin{proof}
Let 
$$
 \mathfrak{E}_1=(E^{\bullet}, A^{E^{\bullet}\prime\prime}+\epsilon_1(t)+\dbar_{\Delta})
$$
and
$$
 \mathfrak{E}_2=(E^{\bullet}, A^{E^{\bullet}\prime\prime}+\epsilon_2(t)+\dbar_{\Delta}).
$$
By Proposition \ref{prop: equivalence and regular equivalence}, we can assume that $\theta$ is a regular equivalence, i.e. $
\theta=\theta(t)$,
 where for each $t\in \Delta$, 
$$
\theta(t): (E^{\bullet}, A^{E^{\bullet}\prime\prime}+\epsilon_1(t))\to (E^{\bullet}, A^{E^{\bullet}\prime\prime}+\epsilon_2(t))
$$
 is a homotopy equivalence of cohesive modules on $X$, and $\theta(t)$ is holomorphic with respect to $t$.

Moreover, by \eqref{eq: deformation restrict to 0} and Definition \ref{defi: strong deformation}, we know that $\theta(0)$ is cochain homotopic to $\id_{E^{\bullet}}$, i.e. we can find an
$$
h\in L_{\mathcal{E}}^{-1}=\bigoplus_{i+j=-1}C^{\infty}(X,\wedge^{i}\overline{T^{*}X}  \hat{\otimes}  \End^{j}(E^{\bullet}))
$$
such that $
\theta(0)=\id_{E^{\bullet}}+\diff_{\mathcal{E}}(h)$. Therefore we can write $\theta$ as 
\begin{equation}\label{eq: decompose of theta with respect to t}
\theta=\id_{E^{\bullet}}+\diff_{\mathcal{E}}(h)+\tau(t)
\end{equation}
where $\tau(t)$ is holomorphic with respect to $t$ and $\tau(0)=0$. Here recall that
$$
\diff_{\mathcal{E}}(h)=A^{E^{\bullet}\prime\prime}h+hA^{E^{\bullet}\prime\prime}.
$$

Now we define 
\begin{equation}\label{eq: decompose of theta tilde with respect to t}
\tilde{\theta}:=\id_{E^{\bullet}}-\epsilon_2(t)h-h\epsilon_1(t)+\tau(t).
\end{equation}
\eqref{eq: decompose of theta with respect to t} and \eqref{eq: decompose of theta tilde with respect to t} give
\begin{equation}
\theta=\tilde{\theta}+(A^{E^{\bullet}\prime\prime}+\epsilon_2(t))h+h(A^{E^{\bullet}\prime\prime}+\epsilon_1(t)),
\end{equation}
i.e $\tilde{\theta}$ is cochain homotopic to $\theta$. Moreover $\tilde{\theta}$ is a strong equivalence between strong deformations.
\end{proof}

The following property of strong equivalences is not difficult to obtain.

\begin{lemma}\label{lemma: strong equivalence is invertible}
A strong equivalence $\theta: \mathfrak{E}_1\to \mathfrak{E}_2$ between two strong deformations of $\mathcal{E}$ over $\Delta$ is invertible if $\Delta$ is sufficiently small. 
\end{lemma}
\begin{proof}
For $\theta=\id_{E^{\bullet}} +\vartheta(t)$. Since $\vartheta(0)=0$, $\theta^{-1}$ is given by
\begin{equation}
\id_{E^{\bullet}}-\vartheta(t)+\vartheta^2(t)-\vartheta^3(t)+\ldots
\end{equation}
on a small neighborhood of $0$. It is clear that $\theta^{-1}$ is still $C^{\infty}$ on $X\times \Delta$ and is holomorphic on the $\Delta$ direction.
\end{proof}

Now we are ready to state the following theorem.

\begin{thm}\label{thm: deformation and Maurer-Cartan}
Let $\mathcal{E}$ be a cohesive module on a compact complex manifold $X$ and  $\Delta$ be  a small ball centered at  $0\in\mathbb{C}^{m}$ as before. Let $L_{\mathcal{E}}$ be the DGLA defined in \eqref{eq: dgla LE}. Then we have
\begin{enumerate}
\item Any deformation of $\mathcal{E}$ corresponds to a Maurer-Cartan element of $L_{\mathcal{E}}\otimes \mathcal{M}(\Delta)$ when $\Delta$ is sufficiently small;
\item Two deformations of $\mathcal{E}$ are equivalent if and only if their corresponding Maurer-Cartan elements are gauge equivalent in $L_{\mathcal{E}}\otimes \mathcal{M}(\Delta)$;
\end{enumerate}
i.e., there is a bijection
\begin{equation}\label{eq: 1-1 correspondence between deformations and Maurer-Cartan}
\text{Def}_{\Delta}(\mathcal{E}) \overset{1\text{-}1}{\leftrightarrow} \text{MC}(L_{\mathcal{E}}\otimes \mathcal{M}(\Delta)).
\end{equation}
In particular, the gauge equivalent class of the Maurer-Cartan element corresponding to a deformation does not depend on the choice of the homotopy inverse and cochain homotopy operators in Proposition \ref{prop: deformation and strong deformation}.
\end{thm}
\begin{proof}
Claim 1 is clear from  Proposition \ref{prop: deformation and strong deformation}. 

To see Claim 2, let $\theta: \mathfrak{F}\to \mathfrak{G}$ be an equivalence between deformations. Let $\xi_1:  \mathfrak{F}\to \mathfrak{E}_1$ and $\xi_2:  \mathfrak{F}\to \mathfrak{E}_2$ be equivalences constructed in Proposition \ref{prop: deformation and strong deformation}, where $\mathfrak{E}_1$ and $\mathfrak{E}_2$ are strong deformations, since all equivalences are invertible up to homotopy, there exists an equivalence
$$
\tilde{\theta}:  \mathfrak{E}_1\to \mathfrak{E}_2.
$$
By Proposition \ref{prop: equivalence and strong equivalence} and Lemma \ref{lemma: strong equivalence is invertible}, we know that $\tilde{\theta}$ is cochain homotopic to $\id_{E^{\bullet}}+\vartheta(t)$ where $\vartheta(0)=0$. Now let 
\begin{equation}
u(t):=\log(\id_{E^{\bullet}}+\vartheta(t))=\vartheta(t)-\frac{\vartheta(t)^2}{2}+\frac{\vartheta(t)^3}{3}-\ldots
\end{equation}
For $\Delta$ sufficiently small, we know that $u(t)$ is $C^{\infty}$ on $X\times \Delta$ and is holomorphic on the $\Delta$ direction. Moreover we have $u(0)=0$. By definition, $u(t)$ satisfies \eqref{eq: gauge equivalent}.

On the other hand, it is obvious that gauge equivalent Maurer-Cartan elements give equivalent strong deformations, hence the corresponding deformations are also equivalent.
\end{proof}

\begin{rmk}
If $\mathcal{E}_1$ and $\mathcal{E}_2$ are homotopy equivalent in $B(X)$. Then   there is a bijection between $\text{Def}_{\Delta}(\mathcal{E}_1)$ and $\text{Def}_{\Delta}(\mathcal{E}_2)$  as in the proof of Lemma \ref{lemma: homotopy equivalence and deformations}. On the other hand there is also a bijection between $\text{MC}(L_{\mathcal{E}_1}\otimes \mathcal{M}(\Delta))$ and $\text{MC}(L_{\mathcal{E}_2}\otimes \mathcal{M}(\Delta))$ by Lemma \ref{lemma: homotopy equivalent L infty equivalent of dgla} and \cite[Theorem 8.13]{yekutieli2012mc}. It is easy to see that these bijections are compatible with \eqref{eq: 1-1 correspondence between deformations and Maurer-Cartan}.
\end{rmk}

\begin{rmk}
We can consider an $\infty$-categorical generalization of Theorem \ref{thm: deformation and Maurer-Cartan}.
In more details, for a DGLA $\mathfrak{g}$, we can associate an $\infty$-groupoid $\Sigma_{\mathfrak{g}}$ as in \cite{hinich2001dg} which enhances MC$(\mathfrak{g})$. It is expected that we can consider all deformations of $\mathcal{E}$ as an $\infty$-groupoid $\mathcal{D}\text{ef}_{\Delta}(\mathcal{E})$ and show that $\mathcal{D}\text{ef}_{\Delta}(\mathcal{E})$ is weak equivalent to $\Sigma_{L_{\mathcal{E}}\otimes  \mathcal{M}(\Delta)}$. This is a topic for future studies.
\end{rmk}

\section{Infinitesimal Deformations}\label{section: infinitesimal deformations}
Let $\mathcal{E}$ be a cohesive module on a compact complex manifold $X$ and  $\Delta$ be a a small ball centered at the origin in $\mathbb{C}^{m}$ as before. For a holomorphic vector $v\in T_{0}\Delta$, let $U$ be a small neighborhood of $0\in \mathbb{C}$ and $f: U\to \Delta$ be a $1$-parameter holomorphic curve in $\Delta$  such that $f(0)=0$ and
$$
\frac{\partial f}{\partial s}\Big|_{s=0}=v.
$$

Let $\mathfrak{F}$ be a family of deformation of $\mathcal{E}$ over $\Delta$ and $\mathfrak{E}$ be a corresponding strong deformation given by Proposition \ref{prop: deformation and strong deformation}. Then the pull back of $\mathfrak{E}$ by $f$ gives a strong deformation  $f^*\mathfrak{E}$ over $U$. 

In more details we have
\begin{equation}
f^*\mathfrak{E}=(E^{\bullet}, A^{E^{\bullet}\prime\prime}+\epsilon(f(s))+\dbar_U)
\end{equation}
where $\epsilon(f(s))$ is a Maurer-Cartan element in $L_{\mathcal{E}}\otimes \mathcal{M}(U)$, i.e. it is a holomorphic function with variable $s$ and valued in $\bigoplus_{i+j=1}C^{\infty}(X,\wedge^{i}\overline{T^{*}X}\hat{\otimes} \End^j(E^{\bullet}))$ which satisfies $\epsilon(f(0))=0$ and
\begin{equation}\label{eq: Maurer-Cartan for epsilon(f(s))}
\diff_{\mathcal{E}}(\epsilon(f(s)))+\frac{1}{2}[\epsilon(f(s)),\epsilon(f(s))]=0.
\end{equation}

Since $\epsilon(f(s))$ is holomorphic, we have its power series expansion
\begin{equation}
\epsilon(f(s))=\epsilon_1 s+\epsilon_2 s^2+\epsilon_3 s^3+\ldots
\end{equation}
Therefore \eqref{eq: Maurer-Cartan for epsilon(f(s))} becomes
\begin{equation}\label{eq: Maurer-Cartan for epsilon 1}
\diff_{\mathcal{E}}(\epsilon_1)=0
\end{equation}
and
\begin{equation}\label{eq: Maurer-Cartan for epsilon l}
\diff_{\mathcal{E}}(\epsilon_l)+\frac{1}{2}\sum_{i+j=l}[\epsilon_i,\epsilon_j]=0, \text{ for }l\geq 2.
\end{equation}
Equation \eqref{eq: Maurer-Cartan for epsilon 1} tells us that $\epsilon_1$ is a degree $1$ closed morphism from $\mathcal{E}$ to itself in the dg-category $B(X)$. Let $[\epsilon_1]$ denote its cocahin homotopy class. It is clear that 
\begin{equation}\label{eq: epsilon 1}
[\epsilon_1]\in \Hom_{\underline{B}(X)}(\mathcal{E},\mathcal{E}[1]).
\end{equation}

\begin{defi}\label{defi: infinitesimal deformation}
We call an element in $\Hom_{\underline{B}(X)}(\mathcal{E},\mathcal{E}[1])$ an \emph{infinitesimal deformation} of $\mathcal{E}$.
\end{defi}

The above construction shows that an genuine deformation gives an infinitesimal deformation. Moreover we have the following result.

\begin{prop}\label{prop: deformation gives infinitesimal deformation}
We have a well defined Kodaira-Spencer type map
\begin{equation}
\KS: T_{0}\Delta\times \text{Def}_{\Delta}(\mathcal{E})\to \Hom_{\underline{B}(X)}(\mathcal{E},\mathcal{E}[1])
\end{equation}
 which is linear on the first component, such that for $v\in T_{0}\Delta$ and $\mathfrak{F}\in  \text{Def}_{\Delta}(\mathcal{E})$, we have $\KS(v,\mathfrak{F}):=[\epsilon_1]$ as in \eqref{eq: epsilon 1}.
\end{prop}
\begin{proof}
We need to check that the cochain homotopy class of $[\epsilon_1]$ is independent of the choices of the holomorphic curve $f$, the representative of equivalent deformations $\mathfrak{F}$, and the strong deformation $\mathfrak{E}$. All of these are easy consequences of Theorem \ref{thm: deformation and Maurer-Cartan} and the fact that gauge equivalence does not change the cochain homotopy class of $[\epsilon_1]$.

The linearity on the first component is clear from the construction.
\end{proof}

\begin{rmk}
The construction in this section generalized the classical result that infinitesimal deformations of a coherent sheaf $\mathcal{F}$ on a scheme $X$ are in $1$-$1$ correspondence with elements in $\Ext^1_X(\mathcal{F},\mathcal{F})$ as in \cite[Theorem 2.7]{hartshorne2010deformation}. 
\end{rmk}

\section{Analytic theory of deformations of cohesive modules}\label{section: analytic theory}
\subsection{Graded Hermitian metrics and Laplacians}\label{subsection: metrics and Laplacians}
We can apply analytic method as in \cite[Section 7]{kobayashi2014differential} to  further study  deformations of cohesive modules.

Let $X$ be a compact complex manifold. We fix a Hermitian (not necessarily K\"{a}hler) metric $g$ on $X$. Let $g^{\wedge\overline{T^{*}X}}$ be the induced metric on $\wedge\overline{T^{*}X}$. Let $\mu_X$ be the associated volume form on $X$.

\begin{defi}\label{defi: Hermitian metric on cohesive module}
Let  $\mathcal{E}=(E^{\bullet}, A^{E^{\bullet}\prime\prime})$ be a cohesive module on a compact complex manifold $X$. A graded Hermitian metric on $\mathcal{E}$ is a metric $h$ on $E^{\bullet}$ such that $h|_{E^i}$ is Hermitian for each $i$ and $h(E^i,E^j)=0$ for $i\neq j$.

The graded Hermitian metric $h$ naturally induces a graded Hermitian metric on $\End^{\bullet}(E^{\bullet})$, which we still denote by $h$.
\end{defi}

\begin{rmk}
A graded Hermitian metric is called a pure metric in  \cite[Definition 4.4.1]{bismut2021coherent}.
\end{rmk}

We equip $\wedge^{\bullet}\overline{T^{*}X}\hat\otimes \End^{\bullet}(E^{\bullet})$ with the metric $g^{\wedge\overline{T^{*}X}}\otimes h$. In more details, for $\alpha=\sum s_i t_i$ and $\beta=\sum u_j w_j$ where $s_i, u_j\in \wedge^{\bullet}\overline{T^{*}X}$ and $t_i, w_j\in \End^{\bullet}(E^{\bullet})$, we have
\begin{equation}
\langle \alpha, \beta \rangle_{g^{\wedge\overline{T^{*}X}}\otimes h}:=\sum_{i,j} g^{\wedge\overline{T^{*}X}}(s_i, u_j) h(t_i, w_j).
\end{equation}

\begin{defi}
We define a global inner product $( \cdot, \cdot )_h$ on 
$$
L_{\mathcal{E}}^{\bullet}=\bigoplus_{a+b=\bullet}C^{\infty}(X,\wedge^{a}\overline{T^{*}X}\hat\otimes \End^b(E^{\bullet}))
$$
 by 
\begin{equation}
(\alpha, \beta )_h:=\int_X \langle \alpha, \beta \rangle_{g^{\wedge\overline{T^{*}X}}\otimes h} \mu_X.
\end{equation}
\end{defi}

Recall that we have the differential $\diff_{\mathcal{E}}$ on  $L_{\mathcal{E}}^{\bullet}$ induced by  $A^{E^{\bullet}\prime\prime}$. Let $\diff_{\mathcal{E}}^*$ be the adjoint of $\diff_{\mathcal{E}}$ with respect to $( \cdot, \cdot )_h$, i.e.
\begin{equation}
(\diff_{\mathcal{E}}\alpha, \beta)_h=(\alpha, \diff_{\mathcal{E}}^*\beta)_h.
\end{equation}
In more details, recall that in \eqref{eq: decomposition of anti super conn} we have the decomposition  
$$
A^{E^{\bullet}\prime\prime}=v_0+\nabla^{E^{\bullet}\prime\prime}+v_2+\ldots.
$$ 
We define
 $$
\nabla^{E^{\bullet}\prime\prime,*}: C^{\infty}(X,\wedge^{\bullet}\overline{T^{*}X}\hat\otimes \End^{\bullet}(E^{\bullet})) \to C^{\infty}(X,\wedge^{\bullet-1}\overline{T^{*}X}\hat\otimes \End^{\bullet}(E^{\bullet}))
$$ to be the unique operator such that
\begin{equation}
(\nabla^{E^{\bullet}\prime\prime}\alpha, \beta)_h=(\alpha, \nabla^{E^{\bullet}\prime\prime,*}\beta)_h;
\end{equation}
and for $i\neq 1$ we define
$$
v_i^*: C^{\infty}(X,\wedge^{\bullet}\overline{T^{*}X}\hat\otimes \End^{\bullet}(E^{\bullet})) \to C^{\infty}(X,\wedge^{\bullet-i}\overline{T^{*}X}\hat\otimes \End^{\bullet+i-1}(E^{\bullet}))
$$
to be the unique operator such that 
\begin{equation}
(v_i\alpha, \beta)_h=(\alpha, v_i^*\beta)_h.
\end{equation}
It is clear that $\nabla^{E^{\bullet}\prime\prime,*}$ is a first order differential operator and $v_i^*$ is $C^{\infty}(X)$-linear.

For $\phi\in L_{\mathcal{E}}^{k}$ with the decomposition $
\phi=\phi_0+\phi_1+\ldots$
as in \eqref{eq: decompose of phi}. Recall that \eqref{eq: diff E in LE shorter} tells us that the $l$th component of $\diff_{\mathcal{E}}\phi$ is
$$
(\diff_{\mathcal{E}}\phi)_l=\nabla^{E^{\bullet}\prime\prime}(\phi_{l-1})+\sum_{i\neq 1}[v_i, \phi_{l-i}].
$$
Then we have
\begin{equation}\label{eq: diff E star in LE shorter} 
(\diff_{\mathcal{E}}^*\phi)_l=\nabla^{E^{\bullet}\prime\prime, *}(\phi_{l+1})+\sum_{i\neq 1}[v_i^*, \phi_{l+i}].
\end{equation}
Similar to \eqref{eq: supercommutator in component}, if we write
$$
v_i=\omega\hat{\otimes} M\in C^{\infty}(X,\wedge^{i}\overline{T^{*}X}\hat\otimes \End^{1-i}(E^{\bullet})) 
$$
and
$$
\phi_{l+i}=\mu\hat{\otimes} N\in C^{\infty}(X,\wedge^{l+i}\overline{T^{*}X}\hat\otimes \End^{-i}(E^{\bullet})).
$$
Let $M^*\in  \End^{i-1}(E^{\bullet})$ be the adjoint of $M$ under $h$. Then we have
\begin{equation}
[v_i^*, \phi_{l+i}]:=(-1)^{(1-i)(l+i)}\iota_{\omega}\mu\hat{\otimes}[M^*,N].
\end{equation}

We can define the Laplacian $
\square^{\mathcal{E}}_h: L_{\mathcal{E}}^{\bullet}\to L_{\mathcal{E}}^{\bullet}$ as
\begin{equation}\label{eq: Laplacian}
\square^{\mathcal{E}}_h:=\diff_{\mathcal{E}}^*\diff_{\mathcal{E}}+\diff_{\mathcal{E}}\diff_{\mathcal{E}}^*.
\end{equation}

\begin{lemma}\label{lemma: Laplacian is elliptic}
The Laplacian $\square^{\mathcal{E}}_h$ is a second order elliptic differential operator.
\end{lemma}
\begin{proof}
The highest order component of $\square^{\mathcal{E}}_h$ is $\nabla^{E^{\bullet}\prime\prime, *}\nabla^{E^{\bullet}\prime\prime}+\nabla^{E^{\bullet}\prime\prime}\nabla^{E^{\bullet}\prime\prime, *}$. The claim is then obvious.
\end{proof}

\subsection{Sobolev spaces and the Green operator}\label{subsection: Sobolev spaces}
So far we have only considered $C^{\infty}$-sections of $\wedge^{\bullet}\overline{T^{*}X}\hat\otimes \End^{\bullet}(E^{\bullet})$. For later purpose we need to consider Sobolev sections too. 
In more details, we pick and fix a $\diff$-connection $\nabla$ on the graded vector bundle  $\wedge^{\bullet}\overline{T^{*}X}\hat{\otimes}\End^{\bullet}(E^{\bullet})$.  The metric $g$ and $h$ induce a metric on the graded bundle $\wedge^{\bullet}T^{*}_\mathbb{C}X\hat{\otimes}\wedge^{\bullet}\overline{T^{*}X}\hat{\otimes}\End^{\bullet}(E^{\bullet})$.  Therefore for an integer $k\geq 0$, we define the $W_{k}$-norm $\lVert \cdot\rVert_{k}$ as
\begin{equation}
\lVert v\rVert_{k}:=(\sum_{i\leq k}\int_X \lVert \nabla^i v\rVert^2 \mu_X)^{\frac{1}{2}}.
\end{equation}
Let $W_{L_{\mathcal{E}}^{\bullet}}^k$ be the completion of  $C^{\infty}(X,\wedge^{\bullet}\overline{T^{*}X}\hat\otimes \End^{\bullet}(E^{\bullet}))$ under the norm $\lVert \cdot\rVert_{k}$. We know that $W_{L_{\mathcal{E}}^{\bullet}}^k$ is a Hilbert space and we have the natural inclusions:
\begin{equation}
L_{\mathcal{E}}^{\bullet}\subset\ldots \subset W_{L_{\mathcal{E}}^{\bullet}}^k \subset W_{L_{\mathcal{E}}^{\bullet}}^{k-1}\subset\ldots \subset W_{L_{\mathcal{E}}^{\bullet}}^0
\end{equation}
where the inclusion map $ W_{L_{\mathcal{E}}^{\bullet}}^k \subset W_{L_{\mathcal{E}}^{\bullet}}^{k-1}$ is compact. Moreover, we have
\begin{equation}
L_{\mathcal{E}}^{\bullet}=\bigcap_{k=0}^{\infty} W_{L_{\mathcal{E}}^{\bullet}}^k 
\end{equation}

Since $\square^{\mathcal{E}}_h$ is a second order differential operator, it extends to a continuous map
\begin{equation}
\square^{\mathcal{E}}_h:W_{L_{\mathcal{E}}^{\bullet}}^k\to W_{L_{\mathcal{E}}^{\bullet}}^{k-2}.
\end{equation}
Since $\square^{\mathcal{E}}_h$ is elliptic, we know that
\begin{equation}
L_{\mathcal{E}}^{\bullet}\cap \ker \square^{\mathcal{E}}_h=W_{L_{\mathcal{E}}^{\bullet}}^k\cap \ker \square^{\mathcal{E}}_h, \text{ for each }k.
\end{equation}

For each $i$, let 
\begin{equation}
\mathbf{H}L_{\mathcal{E}}^i:=L_{\mathcal{E}}^i\cap \ker \square^{\mathcal{E}}_h
\end{equation}
be the subspace of harmonic forms in $L_{\mathcal{E}}^i$.
Therefore we have the orthogonal projection $
H: W_{L_{\mathcal{E}}^{\bullet}}^0 \to \mathbf{H} L_{\mathcal{E}}^{\bullet}
$.  It is clear that
\begin{equation}\label{eq: harmonic forms d d star}
\mathbf{H}L_{\mathcal{E}}^i=\ker\diff_{\mathcal{E}}\cap \ker\diff_{\mathcal{E}}^*\cap L_{\mathcal{E}}^i.
\end{equation}

\begin{rmk}
In general $\id_{E^{\bullet}}$ is in $\ker\diff_{\mathcal{E}}\cap L_{\mathcal{E}}^0$ but is not necessarily in $\mathbf{H}L_{\mathcal{E}}^0$.
\end{rmk}

 Let $(\mathbf{H}L_{\mathcal{E}}^{\bullet})^{\bot}$ be the orthogonal complement of $\mathbf{H}L_{\mathcal{E}}^{\bullet}$ in  $W_{L_{\mathcal{E}}^{\bullet}}^0$. Then we have
\begin{equation}
L_{\mathcal{E}}^{\bullet}=\mathbf{H}L_{\mathcal{E}}^{\bullet}\oplus ((\mathbf{H}L_{\mathcal{E}}^{\bullet})^{\bot}\cap L_{\mathcal{E}}^{\bullet}).
\end{equation}

Again since $\square^{\mathcal{E}}_h$ is elliptic, we can define the Green operator  $G: W_{L_{\mathcal{E}}^{\bullet}}^k\to W_{L_{\mathcal{E}}^{\bullet}}^{k+2}$ which is continuous and satisfies
\begin{equation}\label{eq: Hodge projection and Green operator}
\id_{W_{L_{\mathcal{E}}^{\bullet}}^0}=H+\square^{\mathcal{E}}_h G
\end{equation}
and
\begin{equation}\label{eq: Green operator commutes}
G \diff_{\mathcal{E}}=\diff_{\mathcal{E}}G, ~G\diff_{\mathcal{E}}^*=\diff_{\mathcal{E}}^*G, ~\square^{\mathcal{E}}_h G=G\square^{\mathcal{E}}_h.
\end{equation}
In particular, for each $k$, there exists a constant $c>0$ such that 
\begin{equation}\label{eq: bounded operators}
\lVert Gv\rVert_{k+2}\leq c \lVert v\rVert_{k} \text{ and } \lVert \diff_{\mathcal{E}}^* Gv\rVert_{k+1}\leq c \lVert v\rVert_{k}
\end{equation}
for any $v\in L_{\mathcal{E}}^{\bullet}$, and 
\begin{equation}\label{eq: bounded commutators}
\lVert [u,v]\rVert_{k+1}\leq c \lVert u\rVert_{k} \lVert v\rVert_{k}
\end{equation}
for any $u$, $v\in L_{\mathcal{E}}^{\bullet}$.

Equation \eqref{eq: Hodge projection and Green operator} gives the orthogonal decomposition
\begin{equation}\label{eq: decomposition of L E}
L_{\mathcal{E}}^{\bullet}=\text{Im }\diff_{\mathcal{E}}\oplus \text{Im }\diff_{\mathcal{E}}^*\oplus  \mathbf{H} L_{\mathcal{E}}^{\bullet}.
\end{equation}
Then by \eqref{eq: DGLA LE and morphism in B(X)} we have $\mathbf{H}L_{\mathcal{E}}^i\cong \Hom_{\underline{B}(X)}(\mathcal{E},\mathcal{E}[i])$ for each $i$.

\subsection{The slice space}\label{subsection: the slice space}
\begin{defi}\label{defi: slice in Maurer-Cartan}
We consider the \emph{slice space} $\mathcal{S}_{\diff_{\mathcal{E}}}$ defined by
\begin{equation}\label{eq: slice in Maurer-Cartan}
\mathcal{S}_{\diff_{\mathcal{E}}}:=\{\alpha\in  L_{\mathcal{E}}^1| \diff_{\mathcal{E}}\alpha+\frac{1}{2}[\alpha,\alpha]=0 \text{ and }\diff_{\mathcal{E}}^*\alpha=0\}.
\end{equation}
\end{defi}
It is clear that $0\in \mathcal{S}_{\diff_{\mathcal{E}}}$. Roughly speaking, the condition $\diff_{\mathcal{E}}^*\alpha=0$ means that $\mathcal{S}_{\diff_{\mathcal{E}}}$ is perpendicular to the orbit of the gauge equivalence on the set of Maurer-Cartan elements.

We can make the above picture more precise. Recall that we have  completed tensor products $\mathcal{S}_{\diff_{\mathcal{E}}}\otimes \mathcal{M}(\Delta)$ and $L_{\mathcal{E}}\otimes \mathcal{M}(\Delta)$ as in  Remark \ref{rmk: the topological tensor product}. Also recall we denote the set of gauge equivalent classes of Maurer-Cartan elements in $ L_{\mathcal{E}}\otimes \mathcal{M}(\Delta)$ by MC$(L_{\mathcal{E}}\otimes \mathcal{M}(\Delta))$.

\begin{prop}\label{prop: slice is bijective}
The natural map
\begin{equation}\label{eq: slice to Maurer-Cartan}
\begin{split}
p: \mathcal{S}_{\diff_{\mathcal{E}}}\otimes \mathcal{M}(\Delta)&\to \text{MC}(L_{\mathcal{E}}\otimes \mathcal{M}(\Delta))\\
\alpha(t)&\mapsto [\alpha(t)]
\end{split}
\end{equation}
is surjective for $\Delta$ sufficiently small.
\end{prop}
\begin{proof}
The proof is similar to that of \cite[Theorem 7.3.17]{kobayashi2014differential}.
Let $\beta(t)\in L_{\mathcal{E}}^1\otimes \mathcal{M}(\Delta)$ be a Maurer-Cartan element. We need to prove that there exists an  $u(t)\in L_{\mathcal{E}}^0\otimes \mathcal{M}(\Delta)$ such that
\begin{equation}
e^{-u(t)}\circ (\diff_{\mathcal{E}}+\beta(t))\circ e^{u(t)}-\diff_{\mathcal{E}}\in \mathcal{S}_{\diff_{\mathcal{E}}}.
\end{equation}
It is clear that $e^{-u}\circ (\diff_{\mathcal{E}}+\beta(t))\circ e^{u}-\diff_{\mathcal{E}}$ is still a Maurer-Cartan element for any $u$, so it remains to prove that  we can find a $u(t)$ such that
\begin{equation}
\diff_{\mathcal{E}}^*(e^{-u(t)}\circ (\diff_{\mathcal{E}}+\beta(t))\circ e^{u(t)}-\diff_{\mathcal{E}})=0.
\end{equation}

Let $V$ denote the subspace $\text{Im }\diff_{\mathcal{E}}^*\subset L_{\mathcal{E}}^0$. Hence
\begin{equation}\label{eq: Laplacian is half on V}
\diff_{\mathcal{E}}^*\diff_{\mathcal{E}} v=\square^{\mathcal{E}}_h v \text{ for } v\in V.
\end{equation}

Let $\Omega$ be a small neighborhood of $0$ in $V$. We define a map $F_{\beta}: \Delta\times \Omega\to V$ by
\begin{equation}\label{eq: F V to V}
F_{\beta}(t,v):=\diff_{\mathcal{E}}^*\Big((\id_{E^{\bullet}}+v)^{-1}\circ (\diff_{\mathcal{E}}+\beta(t))\circ (\id_{E^{\bullet}}+v)-\diff_{\mathcal{E}}\Big).
\end{equation}

For $k\geq 0$, we define the Sobolev completions $W_V^k$ and $W_V^{k+2}$ as in Section \ref{subsection: Sobolev spaces}. We can extend the map $F_{\beta}$ in \eqref{eq: F V to V} to a map $F_{\beta}: \Delta\times \tilde{\Omega}\to W_V^k$, 
where $\tilde{\Omega}$ is a small neighborhood of $0$ in $W_V^{k+2}$.

We know that $\beta(0)=0$ hence
\begin{equation}
\begin{split}
F_{\beta}(0,v)=&\diff_{\mathcal{E}}^*\Big((\id_{E^{\bullet}}+v)^{-1}\circ \diff_{\mathcal{E}}\circ (\id_{E^{\bullet}}+v)-\diff_{\mathcal{E}}\Big)\\
=&\diff_{\mathcal{E}}^*\diff_{\mathcal{E}} v\\
=&\square^{\mathcal{E}}_h v
\end{split}
\end{equation} 
where the last equality is obtained by \eqref{eq: Laplacian is half on V}. Hence $F_{\beta}(0,0)=0$ and  the differential of $F_{\beta}$ at $(0,0)$ is
\begin{equation}\label{eq: differential of F at 0}
\frac{\partial F_{\beta}}{\partial v}\big|_{t=0, v=0}=\square^{\mathcal{E}}_h:  W_V^{k+2}\to W_V^k.
\end{equation}
Moreover, by \eqref{eq: Hodge projection and Green operator} and  \eqref{eq: decomposition of L E} we know that 
$\square^{\mathcal{E}}_h:  W_V^{k+2}\to W_V^k$ is invertible. 
Then by the implicit function theorem of Banach spaces \cite[Theorem 7.13-1]{ciarlet2013linear}, for $\Delta$ sufficiently small, there exists a unique $C^{\infty}$ map $v: \Delta\to W_V^{k+2}$ such that $v(0)=0$ and 
 \begin{equation}\label{eq: F(t vt)}
F_{\beta}(t,v(t))=\diff_{\mathcal{E}}^*\Big((\id_{E^{\bullet}}+v(t))^{-1}\circ (\diff_{\mathcal{E}}+\beta(t))\circ (\id_{E^{\bullet}}+v(t))-\diff_{\mathcal{E}}\Big)=0
\end{equation}
 for any $t\in \Delta$.
Moreover, since $\beta(t)$ is a holomorphic function of $t$, we know that $ \frac{\partial F_{\beta}}{\dbar t}\equiv 0$, hence
\begin{equation}
\frac{\partial v}{\dbar t}=-\Big(\frac{\partial F_{\beta}}{\partial v}\Big)^{-1}\circ  \frac{\partial F_{\beta}}{\dbar t} \equiv -\Big(\frac{\partial F_{\beta}}{\partial v}\Big)^{-1}\circ 0=0,
\end{equation}
i.e. $v(t)$ is a holomorphic function of $t$.

We still need to show that $v$ gives a $C^{\infty}$-section of $\bigoplus_{i+j=0}\wedge^{i}\overline{T^{*}X}\hat{\otimes} \End^j(E^{\bullet})$ on $X\times \Delta$. 
Actually
\begin{equation}
\begin{split}
0=F_{\beta(t)}(v(t))=&\diff_{\mathcal{E}}^*\Big((\id_{E^{\bullet}}+v(t))^{-1}\circ (\diff_{\mathcal{E}}+\beta(t))\circ (\id_{E^{\bullet}}+v(t))-\diff_{\mathcal{E}}\Big)\\
=&v(t)^{-1}(\diff_{\mathcal{E}}^*\diff_{\mathcal{E}}v(t)+\text{ lower degree terms})\\
=&v(t)^{-1}(\square^{\mathcal{E}}_h v(t)+\text{ lower degree terms}).
\end{split}
\end{equation}
Moreover $v(t)$ is holomophic so $\frac{\partial }{\partial t}\frac{\partial }{\dbar t} v(t)=0$ hence we have 
\begin{equation}
\Big(\square^{\mathcal{E}}_h +\frac{\partial }{\partial t}\frac{\partial }{\dbar t}\Big) v(t)+\text{ lower degree terms}=0.
\end{equation}
Since $\square^{\mathcal{E}}_h +\frac{\partial }{\partial t}\frac{\partial }{\dbar t}$ is a second order elliptic differential operator on $X\times \Delta$, by elliptic regularity we know $v$ is $C^{\infty}$.

Now we let 
\begin{equation}
u:=\log(\id_{E^{\bullet}}+v)
\end{equation}
hence $e^u=\id_{E^{\bullet}}+v$. For $t$ sufficiently close to $0$, we know that $v(t)$ is sufficiently close to $0$ too. Hence $u$ is well-defined and
\begin{equation}
u\in \bigoplus_{i+j=0} C^{\infty}(X\times \Delta, \wedge^{i}\overline{T^{*}X}\hat{\otimes} \End^j(E^{\bullet}))
\end{equation} 
The claim is now a consequence of \eqref{eq: F(t vt)}.
\end{proof}

\begin{rmk}
The injectivity of the map $p$ in \eqref{eq: slice to Maurer-Cartan} is more subtle. We may impose a generalized  \emph{simplicity} condition of $\mathcal{E}$ and prove the injectivity in the more general sense as in \cite[Theorem 7.3.17]{kobayashi2014differential}. This is the topic of  future studies.
\end{rmk}

\subsection{The Kuranishi map}\label{subsection: Kuranishi map}
\begin{lemma}\label{lemma: orthogonal decomposition of Maurer-Cartan equation}
An element $\alpha\in L_{\mathcal{E}}^1$ is a Maurer-Cartan element if and only if it satisfies
\begin{align}
\diff_{\mathcal{E}}(\alpha+\frac{1}{2}\diff_{\mathcal{E}}^*G[\alpha,\alpha])=0, \label{eq: orthogonal decomposition of Maurer-Cartan equation1}\\
\diff_{\mathcal{E}}^*\diff_{\mathcal{E}}G[\alpha,\alpha]=0, \label{eq: orthogonal decomposition of Maurer-Cartan equation2}\\
H[\alpha,\alpha]=0. \label{eq: orthogonal decomposition of Maurer-Cartan equation3}
\end{align}
\end{lemma}
\begin{proof}
Apply \eqref{eq: Hodge projection and Green operator} to $\diff_{\mathcal{E}}\alpha+\frac{1}{2}[\alpha,\alpha]$, we get
\begin{equation}\label{eq: Maurer-Cartan variation1}
\diff_{\mathcal{E}}\alpha+\frac{1}{2}[\alpha,\alpha]=H(\diff_{\mathcal{E}}\alpha+\frac{1}{2}[\alpha,\alpha])+\square^{\mathcal{E}}_h G(\diff_{\mathcal{E}}\alpha+\frac{1}{2}[\alpha,\alpha])
\end{equation}
Since $\square^{\mathcal{E}}_h=\diff_{\mathcal{E}}^*\diff_{\mathcal{E}}+\diff_{\mathcal{E}}\diff_{\mathcal{E}}^*$, by \eqref{eq: Green operator commutes} we get
\begin{equation}\label{eq: Maurer-Cartan variation2}
\square^{\mathcal{E}}_h G(\diff_{\mathcal{E}}\alpha+\frac{1}{2}[\alpha,\alpha])=\diff_{\mathcal{E}}(\diff_{\mathcal{E}}^*G\diff_{\mathcal{E}}\alpha+\frac{1}{2}\diff_{\mathcal{E}}^*G[\alpha,\alpha])+\frac{1}{2}\diff_{\mathcal{E}}^*\diff_{\mathcal{E}}G[\alpha,\alpha].
\end{equation}
Again by \eqref{eq: Hodge projection and Green operator} and \eqref{eq: Green operator commutes} we have
\begin{equation}\label{eq: Maurer-Cartan variation3}
\diff_{\mathcal{E}}^*G\diff_{\mathcal{E}}\alpha=\square^{\mathcal{E}}_h G\alpha-\diff_{\mathcal{E}}\diff_{\mathcal{E}}^*G\alpha=\alpha-H\alpha-\diff_{\mathcal{E}}\diff_{\mathcal{E}}^*G\alpha.
\end{equation}
Combine \eqref{eq: Maurer-Cartan variation1}, \eqref{eq: Maurer-Cartan variation2},   and \eqref{eq: Maurer-Cartan variation3} we get
\begin{equation}\label{eq: Maurer-Cartan variation4}
\diff_{\mathcal{E}}\alpha+\frac{1}{2}[\alpha,\alpha]=\frac{1}{2}H[\alpha,\alpha]+\diff_{\mathcal{E}}(\alpha+\frac{1}{2}\diff_{\mathcal{E}}^*G[\alpha,\alpha])+\frac{1}{2}\diff_{\mathcal{E}}^*\diff_{\mathcal{E}}G[\alpha,\alpha].
\end{equation}

The equivalence of the Maurer-Cartan equation and \eqref{eq: orthogonal decomposition of Maurer-Cartan equation1}, \eqref{eq: orthogonal decomposition of Maurer-Cartan equation2}, \eqref{eq: orthogonal decomposition of Maurer-Cartan equation3} now follows from \eqref{eq: decomposition of L E} and \eqref{eq: Maurer-Cartan variation4}.
\end{proof}

\begin{defi}\label{defi: Kuranishi map}
We define the Kuranishi map $\ku: L_{\mathcal{E}}^1\to L_{\mathcal{E}}^1$ by
\begin{equation}\label{eq: Kuranishi map}
\ku(\alpha)=\alpha+\frac{1}{2}\diff_{\mathcal{E}}^*G[\alpha,\alpha].
\end{equation}
\end{defi}

\begin{rmk}
To simplify the notation, we denote $\ku\otimes \id: L_{\mathcal{E}}^1\otimes \mathcal{M}(\Delta)\to L_{\mathcal{E}}^1\otimes \mathcal{M}(\Delta)$ also by $\ku$. The same convention applies to $H\otimes \id$, $G\otimes \id$, etc.
\end{rmk}

\begin{lemma}\label{lemma: Kuranishi map harmonic form}
If $\alpha\in \mathcal{S}_{\diff_{\mathcal{E}}}$ as in Definition \ref{defi: slice in Maurer-Cartan}, then $\ku(\alpha)\in \mathbf{H} L_{\mathcal{E}}^1$.
\end{lemma}
\begin{proof}
For any Maurer-Cartan element $\alpha$, we have  $\diff_{\mathcal{E}}(\ku(\alpha))=0$ by \eqref{eq: orthogonal decomposition of Maurer-Cartan equation1}. If we have further $\diff_{\mathcal{E}}^*\alpha=0$, then we have 
$\diff_{\mathcal{E}}^*(\ku(\alpha))=0$ by \eqref{eq: orthogonal decomposition of Maurer-Cartan equation2}. The claim now follows from \eqref{eq: harmonic forms d d star}.
\end{proof}

\begin{lemma}\label{lemma: Kuranishi map is surjective on harmonic form}
Let $\Delta\subset \mathbb{C}^m$ be a small neighborhood of $0$. The the map
\begin{equation}
\ku: L_{\mathcal{E}}^1\otimes \mathcal{M}(\Delta)\to \mathbf{H} L_{\mathcal{E}}^1\otimes \mathcal{M}(\Delta) 
\end{equation}
is surjective.
\end{lemma}
\begin{proof}
The argument is similar to that in \cite[Section 2]{kuranishi1965new}  and \cite[Section 5]{chan2016differential}.  Let $\beta(t)\in \mathbf{H} L_{\mathcal{E}}^1\otimes \mathcal{M}(\Delta)$, i.e. $\beta(t)$ is a holomorphic family of elements in $\mathbf{H} L_{\mathcal{E}}^1$ such that $\beta(0)=0$.
We fix an integer $k\geq 0$. By \eqref{eq: bounded operators} we know that the Kuranishi map $\ku$ extends to a continuous linear map 
$$
\ku: W_{L_{\mathcal{E}}^{\bullet}}^k\to W_{L_{\mathcal{E}}^{\bullet}}^k.
$$
Consider the map $F_{\beta}: \Delta\times W_{L_{\mathcal{E}}^{\bullet}}^k\to W_{L_{\mathcal{E}}^{\bullet}}^k$ defined by
\begin{equation}
F_{\beta}(t,\alpha)=\ku(\alpha)-\beta(t).
\end{equation}
Then we have $F_{\beta}(0,\alpha)=\ku(\alpha)$ and in particular $F_{\beta}(0,0)=0$.
We can also check that  
\begin{equation}
\frac{\partial F_{\beta}}{\partial \alpha}\big|_{t=0,\alpha=0}=\id_{W_{L_{\mathcal{E}}^{\bullet}}^k}.
\end{equation} 
Therefore by the implicit function theorem of Banach spaces and the same argument as in the proof of Proposition \ref{prop: slice is bijective}, we obtain a holomorphic family $\alpha(t)$ valued in $W_{L_{\mathcal{E}}^{\bullet}}^k$ such that $\alpha(0)=0$ and  $F_{\beta}(t,\alpha(t))=0$  for any $t\in \Delta$, i.e. we have
\begin{equation}
\ku(\alpha(t))=\beta(t) \text{ for any } t\in \Delta.
\end{equation}

It remains to show that $\alpha(t)$ is $C^{\infty}$. We apply $\square^{\mathcal{E}}_h$ to $\ku (\alpha(t))$ and get
\begin{equation}\label{eq: d star alpha 1}
\square^{\mathcal{E}}_h\alpha(t)+\frac{1}{2}\square^{\mathcal{E}}_h\diff_{\mathcal{E}}^*G[\alpha(t),\alpha(t)]=\square^{\mathcal{E}}_h(\beta(t))= 0.
\end{equation}
Since $\id_{W_{L_{\mathcal{E}}^{\bullet}}^0}=H+\square^{\mathcal{E}}_h G$ and $\diff_{\mathcal{E}}^*G=G\diff_{\mathcal{E}}^*$, we know that
\begin{equation}\label{eq: d star alpha 2}
\begin{split}
\square^{\mathcal{E}}_h\diff_{\mathcal{E}}^*G[\alpha(t),\alpha(t)]&=\diff_{\mathcal{E}}^*[\alpha(t),\alpha(t)]-H\diff_{\mathcal{E}}^*[\alpha(t),\alpha(t)]\\
&=\diff_{\mathcal{E}}^*[\alpha(t),\alpha(t)]
\end{split}
\end{equation}
as $H|_{\text{Im}\diff_{\mathcal{E}}^*}=0$. Combine \eqref{eq: d star alpha 1} and \eqref{eq: d star alpha 2} we get
\begin{equation}
\square^{\mathcal{E}}_h\alpha(t)+\frac{1}{2}\diff_{\mathcal{E}}^*[\alpha(t),\alpha(t)]=0.
\end{equation}
Moreover, as $\alpha(t)$ is holomorphic with respect to $t$, we have
\begin{equation}
\square^{\mathcal{E}}_h\alpha(t)+\frac{\partial }{\partial t}\frac{\partial }{\dbar t} \alpha(t)+\frac{1}{2}\diff_{\mathcal{E}}^*[\alpha(t),\alpha(t)]=0.
\end{equation}
We notice that the highest order part is $\square^{\mathcal{E}}_h+\frac{\partial }{\partial t}\frac{\partial }{\dbar t}$, whose ellipticity implies that 
\begin{equation}
\alpha\in\bigoplus_{i+j=1} C^{\infty}(X\times \Delta, \wedge^{i}\overline{T^{*}X}\hat{\otimes} \End^j(E^{\bullet})).
\end{equation}
\end{proof}

\begin{rmk}
In general, the solution $\alpha(t)$ in Lemma \ref{lemma: Kuranishi map is surjective on harmonic form} may not be in $\mathcal{S}_{\diff_{\mathcal{E}}}$.
\end{rmk}

\subsection{The obstruction}\label{subsection: obstruction}

\begin{prop}\label{prop: obstruction}
Let $\Delta\subset \mathbb{C}^m$ be a small neighborhood of $0$. Let $\beta(t)$ be  in $\mathbf{H} L_{\mathcal{E}}^1\otimes \mathcal{M}(\Delta)$. Let $\alpha(t)\in L_{\mathcal{E}}^1\otimes \mathcal{M}(\Delta)$ be a  solution of $\ku(\alpha(t))=\beta(t)$ as in Lemma \ref{lemma: Kuranishi map is surjective on harmonic form}. For  $\Delta$ sufficiently small, the following are equivalent.
\begin{enumerate}
\item $\alpha(t)\in \mathcal{S}_{\diff_{\mathcal{E}}}\otimes \mathcal{M}(\Delta)$;
\item $\alpha(t)$ is an Maurer-Cartan element for any $t\in \Delta$;
\item $H[\alpha(t),\alpha(t)]=0$ for any $t\in \Delta$.
\end{enumerate}
\end{prop}
\begin{proof}
$1\Rightarrow 2$ is trivial. By \eqref{eq: orthogonal decomposition of Maurer-Cartan equation3} in Lemma \ref{lemma: orthogonal decomposition of Maurer-Cartan equation} we know $2\Rightarrow 3$. We need to prove $3\Rightarrow 1$. 

Since $\ku(\alpha(t))=\alpha(t)+\frac{1}{2}\diff_{\mathcal{E}}^*G[\alpha(t),\alpha(t)]\in \mathbf{H} L_{\mathcal{E}}^1\otimes \mathcal{M}(\Delta)$ , we get
\begin{equation}\label{eq: harmonic of ku alpha}
\diff_{\mathcal{E}}\Big(\alpha(t)+\frac{1}{2}\diff_{\mathcal{E}}^*G[\alpha(t),\alpha(t)]\Big)=0 \text{ and } \diff_{\mathcal{E}}^*\Big(\alpha(t)+\frac{1}{2}\diff_{\mathcal{E}}^*G[\alpha(t),\alpha(t)]\Big)=0.
\end{equation}
From the second equality we immediately get $\diff_{\mathcal{E}}^*\alpha(t)=0$. It remains  to prove $\alpha(t)$ is a Maurer-Cartan element. Let 
\begin{equation}
\gamma(t)=\diff_{\mathcal{E}}\alpha(t)+\frac{1}{2}[\alpha(t),\alpha(t)].
\end{equation}
We want to show that $\gamma(t)=0$ for $t\in \Delta$.

Since we have  $H[\alpha(t),\alpha(t)]=0$, by \eqref{eq: Maurer-Cartan variation4},  \eqref{eq: harmonic of ku alpha}, and \eqref{eq: Green operator commutes} we have
\begin{equation}
\begin{split}
\gamma(t)&=\frac{1}{2}\diff_{\mathcal{E}}^*\diff_{\mathcal{E}}G[\alpha(t),\alpha(t)]\\
&=\frac{1}{2}\diff_{\mathcal{E}}^*G\diff_{\mathcal{E}}[\alpha(t),\alpha(t)]\\
&=\diff_{\mathcal{E}}^*G[\diff_{\mathcal{E}}\alpha(t),\alpha(t)]\\
&=\diff_{\mathcal{E}}^*G[\gamma(t)-\frac{1}{2}[\alpha(t),\alpha(t)],\alpha(t)].
\end{split}
\end{equation}
Since $[[\alpha(t),\alpha(t)],\alpha(t)]=0$, we get
\begin{equation}
\gamma(t)=\diff_{\mathcal{E}}^*G[\gamma(t),\alpha(t)].
\end{equation}

We fix an integer $k\geq 0$. By the norm estimations \eqref{eq: bounded operators} and \eqref{eq: bounded commutators} we get
\begin{equation}\label{eq: norm of gamma}
\lVert \gamma(t)\rVert_{k}\leq \lVert \gamma(t)\rVert_{k+1}\leq c^2 \lVert \alpha(t)\rVert_{k}\lVert \gamma(t)\rVert_{k}
\end{equation}

We know $\alpha(0)=0$. Now we make $\Delta$ sufficiently small so that $\lVert \alpha(t)\rVert_{k}\leq \frac{1}{2c^2}$ for $t\in \Delta$. Hence \eqref{eq: norm of gamma} implies $\lVert \gamma(t)\rVert_{k}\leq \frac{1}{2}\lVert \gamma(t)\rVert_{k}$. So $\gamma(t)=0$ for $t\in \Delta$. We finish the proof.
\end{proof}

\begin{defi}\label{defi: obstruction}
We call $H[\alpha(t),\alpha(t)]\in \mathbf{H} L_{\mathcal{E}}^2\otimes \mathcal{M}(\Delta)$ the obstruction of deformation of $\beta(t)\in \mathbf{H} L_{\mathcal{E}}^1\otimes \mathcal{M}(\Delta)$.
\end{defi}

\subsection{Unobstructed deformations}\label{subsection: unobstructed}
In general it is difficult to check whether the obstruction vanishes. Nevertheless, there are cases that the obstruction vanishes identically.

Recall (see \cite[Section 4.2]{bismut2021coherent}) that we have a supertrace map 
$$
\str: \End^{\bullet}(E^{\bullet}) \to \mathbb{C}
$$
 define by
\begin{equation}
\str(u)=\begin{cases}
\sum_i (-1)^i\text{Tr}( u|_{E^i}), & \text{ if } u \in \End^0(E^{\bullet});\\
0, & \text{ if } u\in \End^j(E^{\bullet}), ~j\neq 0.
\end{cases}
\end{equation}
We can extend it to a supertrace on $L_{\mathcal{E}}^{\bullet}$,
\begin{equation}
\str: L_{\mathcal{E}}^{\bullet}\to \Omega^{0,\bullet}(X)
\end{equation}
 as follows, for a $\phi\in L_{\mathcal{E}}^k=\bigoplus _i C^{\infty}(X,\wedge^i\overline{T^{*}X}\hat\otimes \End^{k-i}(E^{\bullet}))$, we decompose $\phi$ as
$$
\phi=\phi_0+\phi_1+\ldots
$$
where $\phi_i\in C^{\infty}(X,\wedge^i\overline{T^{*}X}\hat\otimes \End^{k-i}(E^{\bullet}))$ for each $i$. In particular
\begin{equation}
\phi_k\in C^{\infty}(X,\wedge^k\overline{T^{*}X}\hat\otimes \End^0(E^{\bullet})).
\end{equation}
Write $\phi_k=\sum_a s_a u_a$ where $s_a\in \Omega^{0,k}(X)$ and $u_a\in C^{\infty}(X, \End^0(E^{\bullet}))$. Then we define
\begin{equation}
\str \phi :=\sum_a s_a \cdot \str u_a\in  \Omega^{0,k}(X).
\end{equation}
It is clear that $\str$ vanishes on supercommutators.

We can define the $\dbar$-Laplacian on $\Omega^{0,\bullet}(X)$ and hence harmonic forms $\mathbf{H}^{0,\bullet}(X)$. We can also define the Sobolev spaces $W^k_{\Omega^{0,\bullet}(X)}$ and the orthogonal projection $H: W^0_{\Omega^{0,\bullet}(X)}\to \mathbf{H}^{0,\bullet}(X)$ as in Section \ref{subsection: Sobolev spaces}. We can extend $\str$ to a map
\begin{equation}
\str: W^k_{L_{\mathcal{E}}^{\bullet}} \to W^k_{\Omega^{0,\bullet}(X)}.
\end{equation}

\begin{lemma}\label{lemma: supertrace commutes with H}
The supertrace map $\str$ restricts to a map $
\str: \mathbf{H}L_{\mathcal{E}}^{\bullet}\to \mathbf{H}^{0,\bullet}(X)$.
Moreover, the $\str$ commutes with the projection $H$, i.e.
\begin{equation}
\str\circ H=H\circ \str: W^k_{L_{\mathcal{E}}^{\bullet}} \to \mathbf{H}^{0,\bullet}(X).
\end{equation}
\end{lemma}
\begin{proof}
It follows easily from the definition of $\square^{\mathcal{E}}_h$ in Section \ref{subsection: metrics and Laplacians} and the fact that $\str$ vanishes on supercommutators.
\end{proof}

Since $\mathbf{H}L_{\mathcal{E}}^{\bullet}\cong \Hom_{\underline{B}(X)}(\mathcal{E},\mathcal{E}[\bullet])$ and $\mathbf{H}^{0,\bullet}(X)\cong H^{0,\bullet}(X)$, we obtain a supertrace map
\begin{equation}
\str: \Hom_{\underline{B}(X)}(\mathcal{E},\mathcal{E}[\bullet])\to H^{0,\bullet}(X).
\end{equation}

\begin{lemma}\label{lemma: obstruction is traceless}
For each $t$, the obstruction $H[\alpha(t),\alpha(t)]$ as in Definition \ref{defi: obstruction} lies in the kernel of $
\str: \mathbf{H}L_{\mathcal{E}}^2\to \mathbf{H}^{0,2}(X)$.
\end{lemma}
\begin{proof}
The supertrace vanishes on commutators so $\str[\alpha(t),\alpha(t)]=0$. The claim then follows from Lemma \ref{lemma: supertrace commutes with H}.
\end{proof}

\begin{prop}
If the supertrace map $\str: \Hom_{\underline{B}(X)}(\mathcal{E},\mathcal{E}[2])\to H^{0,2}(X)$ is injective, then deformations of $\mathcal{E}$ are unobstructed. More precisely, when the hypothesis is satisfied, for any linear map $\beta: T_0\Delta\to \Hom_{\underline{B}(X)}(\mathcal{E},\mathcal{E}[1])$, there exists a deformation $\mathfrak{F}$ on a small neighborhood of $0$ in $\Delta$, such that $\beta(v)=\KS(v,\mathfrak{F})$, where $\KS$ is the Kodaira-Spencer map in Proposition \ref{prop: deformation gives infinitesimal deformation}.
\end{prop}
\begin{proof}
We can treat the linear map $\beta: T_0\Delta\to \Hom_{\underline{B}(X)}(\mathcal{E},\mathcal{E}[1])$ as a holomorphic map 
$$
\beta(t): \Delta\to \mathbf{H}L_{\mathcal{E}}^1
$$
such that $\beta(0)=0$. Let $\alpha(t)$ be a solution to $\ku(\alpha(t))=\beta(t)$ as in Lemma \ref{lemma: Kuranishi map is surjective on harmonic form}. 

If $\str: \Hom_{\underline{B}(X)}(\mathcal{E},\mathcal{E}[2])\to H^{0,2}(X)$ is injective, then $H[\alpha(t),\alpha(t)]=0$ by Lemma \ref{lemma: obstruction is traceless}. Then by Proposition \ref{prop: obstruction}, on a small neighborhood of $0$ in $\Delta$, $\alpha(t)$ give a strong deformation of $\mathcal{E}$, which we denote by $\mathfrak{F}$.

Now for each $v\in T_0\Delta$, let $f$ be a $1$-parameter curve such that $\frac{\partial f}{\partial s}|_{s=0}=v$ as in Section \ref{section: infinitesimal deformations}. Then we have 
\begin{equation}
\alpha(f(s))+\frac{1}{2}\diff_{\mathcal{E}}^*G[\alpha(f(s)),\alpha(f(s))]=\beta(f(s)).
\end{equation}
Therefore it is clear that the first order term in the Taylor expansion of $\alpha(f(s))$ is $\beta(v)$, i.e. we have $\beta(v)=\KS(v,\mathfrak{F})$.
\end{proof}

We have examples of $\mathcal{E}$ such that $\str: \Hom_{\underline{B}(X)}(\mathcal{E},\mathcal{E}[2])\to H^{0,2}(X)$ is injective.

\begin{eg}\label{eg: unobstructed 1}
 Let   $\mathcal{S}\in D^b_{\coh}(X)$ be an object such that $\Hom_{D^b_{\coh}(X)}(\mathcal{S},\mathcal{S}[2])=0$.  Let  $\mathcal{E}\in B(X)$ be such that $\underline{F}_X(\mathcal{E})\simeq \mathcal{S}$ as in Theorem \ref{thm: equiv of cats}. Then $\str$ is injective as  $\Hom_{\underline{B}(X)}(\mathcal{E},\mathcal{E}[2])=0$, hence the deformation of $\mathcal{E}$ is unobstructed.
\end{eg}

\begin{eg}\label{eg: unobstructed 2}
Now let $X$ be a K3 surface hence $H^{0,2}(X)\cong \mathbb{C}$. Let  $\mathcal{R}\in D^b_{\coh}(X)$ be an object such that 
\begin{equation}
\begin{split}
&\Hom_{D^b_{\coh}(X)}(\mathcal{R},\mathcal{R})\cong\Hom_{D^b_{\coh}(X)}(\mathcal{R},\mathcal{R}[2])\cong\mathbb{C},\\
&Hom_{D^b_{\coh}(X)}(\mathcal{R},\mathcal{R}[1])\cong\mathbb{C}^2.
\end{split}
\end{equation}
 For example, a skyscraper sheaf $\mathcal{O}_x$ satisfies this condition, see \cite[Example 4.12]{macri2008automorphisms}.

As shown in \cite[Section 10.2.1]{huybrechts2016lectures}, by Serre duality, the map 
$$
\str: \Hom_{D^b_{\coh}(X)}(\mathcal{R},\mathcal{R}[2])\to H^{0,2}(X)
$$
 is non-zero, hence it must be injective as both the domain and the range are $1$-dimensional. 

Let  $\mathcal{E}\in B(X)$ be such that $\underline{F}_X(\mathcal{E})\simeq \mathcal{R}$ as in Theorem \ref{thm: equiv of cats}. Then the deformation of $\mathcal{E}$ is unobstructed.
\end{eg}

\begin{rmk}
In Example \ref{eg: unobstructed 1} and Example \ref{eg: unobstructed 2} we do not require $X$ to be projective.
\end{rmk}

\bibliography{deform}{}
\bibliographystyle{alpha}

\end{document}